\documentclass[12pt]{amsart}

\usepackage{amssymb}

\usepackage{graphicx}

\calclayout

\usepackage{pdfsync}

\usepackage[colorlinks=true, bookmarks=true,pdfstartview=FitV, linkcolor=blue, citecolor=blue, urlcolor=blue]{hyperref}

\newtheorem{theorem}{Theorem}[section]
\newtheorem{lemma}[theorem]{Lemma}
\newtheorem{prop}[theorem]{Proposition}
\newtheorem{corollary}[theorem]{Corollary}

\theoremstyle{definition}
\newtheorem{definition}[theorem]{Definition}

\theoremstyle{remark}
\newtheorem{remark}[theorem]{Remark}

\numberwithin{equation}{section}

\allowdisplaybreaks

\newcommand{\R}{\mathbb R}

\newcommand{\M}{\mathcal M}

\newcommand{\subRn}{{{\mathbb R}^n}}

\DeclareMathOperator*{\essinf}{ess\,inf}
\DeclareMathOperator*{\esssup}{ess\,sup}
\DeclareMathOperator{\sgn}{sgn}

\newcommand{\pp}{{p(\cdot)}}
\newcommand{\cpp}{{p'(\cdot)}}
\newcommand{\Lp}{L^{p(\cdot)}}
\newcommand{\Pp}{\mathcal P}
\newcommand{\qq}{{q(\cdot)}}

\newcommand{\Lq}{L^{q(\cdot)}}

\newcommand{\F}{\mathcal{F}}

\newcommand{\rr}{{r(\cdot)}}
\newcommand{\crr}{{r'(\cdot)}}
\newcommand{\Rj}{\mathcal{R}}

\def\Xint#1{\mathchoice
   {\XXint\displaystyle\textstyle{#1}}%
   {\XXint\textstyle\scriptstyle{#1}}%
   {\XXint\scriptstyle\scriptscriptstyle{#1}}%
   {\XXint\scriptscriptstyle\scriptscriptstyle{#1}}%
   \!\int}
\def\XXint#1#2#3{{\setbox0=\hbox{$#1{#2#3}{\int}$}
     \vcenter{\hbox{$#2#3$}}\kern-.5\wd0}}

\def\avgint{\Xint-}

\begin{document}

\title[Extrapolation in variable Lebesgue spaces]
{Extrapolation and weighted norm inequalities in the variable Lebesgue spaces}

\author{David Cruz-Uribe, SFO}
\address{Department of Mathematics, Trinity College}
\email{David.CruzUribe@trincoll.edu}

\author{Li-An Daniel Wang}
\address{Department of Mathematics, Trinity College}
\email{Daniel.Wang@trincoll.edu}

\thanks{Both authors are supported by the Stewart-Dorwart faculty
  development fund at Trinity College,  and the first is also
  supported by NSF grant 1362425.
  The authors would like to thank J.M.~Martell for an enlightening
  conversation on limited range extrapolation.}



\date{August 13, 2014}

\begin{abstract}
We extend the theory of Rubio de Francia extrapolation, including
off-diagonal, limited range and $A_\infty$ extrapolation, to the
weighted variable Lebesgue spaces $\Lp(w)$.   As a consequence we are
able to show that a number of different operators from harmonic
analysis are bounded on these spaces.  The proofs of our extrapolation
results are developed in a way that outlines a general approach to
proving extrapolation theorems on other Banach function spaces.
\end{abstract}

\maketitle

\section{Introduction}

The variable Lebesgue spaces $\Lp$ are a generalization of the
classical Lebesgue spaces, replacing the constant exponent $p$ with an
exponent function $\pp$.  It is a Banach function space with the norm
\begin{equation} \label{eqn:defn-norm}
 \|f\|_\pp = \|f\|_{\Lp} =
\inf \left\{ \lambda > 0 : \int_{\R^n\setminus\R^n_\infty}
  \left(\frac{|f(x)|}{\lambda}\right)^{p(x)}\,dx
+ \|f\|_{L^\infty(\R^n_\infty)} \leq 1 \right\},
\end{equation}
where $\R^n_\infty = \{ x : p(x)=\infty \}$.  These spaces have been
the subject of considerable interest since the early 1990s both as
function spaces with intrinsic interest and for their applications to
problems arising in PDEs and the calculus of variations.  For a
thorough discussion of these spaces and their history, see~\cite{cruz-fiorenza-book,diening-harjulehto-hasto-ruzicka2010}.

Recently there has been interest in extending the theory of
Muckenhoupt $A_p$ weights to this setting.  Recall that given a
non-negative, measurable function $w$, for
$1<p<\infty$, $w\in A_p$ if
\[ [w]_{A_p} = \sup_B \left(\avgint_B w(x)\,dx\right)
\left(\avgint_B w(x)^{1-p'}\,dx\right)^{p-1} < \infty, \]
where the supremum is taken over all balls $B \subset \R^n$, and
$\avgint_B w\,dx = |B|^{-1}\int_B w\,dx$.  We say $w\in A_1$ if
\[ [w]_{A_1} = \sup_B \frac{\avgint_B w(x)\,dx}{\essinf_{x\in B} w(x)}
< \infty. \]
These weights characterize the weighted norm inequalities for the
Hardy-Littlewood maximal operator,
\[ Mf(x) = \sup_B \avgint_B |f(y)|\,dy \cdot \chi_{B}(x).   \]
More precisely, $w\in A_p$, $1<p<\infty$, if and only if $M : L^p(w)\rightarrow
L^p(w)$.    The Muckenhoupt weights
also govern the weighted norm inequalities for a large number of
operators in harmonic analysis, including singular integrals,
commutators and square functions.  For details,
see~\cite{cruz-martell-perezBook, duoandikoetxea01,grafakos08b}.

Weighted norm inequalities for the maximal operator in the variable
Lebesgue spaces were proved in~\cite{cruz-diening-hasto2011,MR2927495,
  diening-harjulehto-hasto-ruzicka2010} (see
also~\cite{diening-hastoPreprint2010}
for related results).  To show the connection with the classical
results we restate them by replacing the weight $w$ by $w^p$ in the
definition of $A_p$.  In this case we say that $w\in A_p$, $1<
p\leq \infty$, if
\[ \sup_B |B|^{-1}\|w\chi_B\|_p \|w^{-1}\chi_B\|_{p'} < \infty, \]
and this is equivalent to the norm inequality
\[ \|(Mf)w\|_p \leq C\|fw\|_p. \]

\begin{remark}
Note that in this formulation the inequality holds in the case
$p=\infty$; this fact is not well-known but was first proved by
Muckenhoupt~\cite{muckenhoupt72}.
\end{remark}

In this form the definition  immediately generalizes to the variable
Lebesgue spaces.  (See below for precise definitions.)   We say that a weight $w$ is in the class $A_\pp$ if
\[ \sup_B |B|^{-1}\|w\chi_B\|_\pp \|w^{-1}\chi_B\|_{\cpp} < \infty. \]
When $\pp$ is log-H\"older continuous ($\pp \in LH$) and is
bounded and bounded above $1$  ($1<p_-\leq p_+<\infty$), then  $w\in
A_\pp$ if and only if
\[ \|(Mf)w\|_\pp \leq C\|fw\|_\pp. \]

\medskip

In this paper we further develop the theory of
weighted norm inequalities on the variable Lebesgue spaces.   We
show that the $A_\pp$ weights govern the weighted norm
inequalities for a wide variety of operators in harmonic analysis,
including singular and fractional integrals and the Riesz transforms
associated to elliptic operators in divergence form.
To do this we show that theory of Rubio de Francia extrapolation holds
in this setting.  As an immediate consequence we prove, with very
little additional work, norm inequalities in weighted $\Lp$
spaces for any operator that satisfies estimates on $L^p(w)$ when $w$
is a Muckenhoupt $A_p$ weight.   The classical theory of extrapolation
is a powerful tool in harmonic analysis:  for a detailed treatment,
see~\cite{cruz-martell-perezBook}.
Extrapolation in the scale of the variable Lebesgue spaces was
originally developed in~\cite{MR2210118} to prove unweighted inequalities (see
also~\cite{cruz-fiorenza-book,cruz-martell-perezBook}).   It has found
wide application since (see, for
instance,~\cite{DCU-dw-P2014,MR2498558,MR3057132,MR2869787}), and the
results we present here should be equally useful.  We note that our work has
already been applied to the study of greedy approximation algorithms
on variable Lebesgue spaces in~\cite{dcu-eh-jmm14}.

The remainder of this paper is organized as follows.  In
Section~\ref{section:main-theorems} we state
our extrapolation results, including the precise definitions needed.
In Section~\ref{section:applications} we show how to apply
extrapolation to prove weighted norm inequalities for several
different kinds of operators.  Our examples are not exhaustive;
rather, they were chosen to illustrate the applicability of
extrapolation.  In Section~\ref{section:general} we give a general
overview of our approach to proving extrapolation theorems.  These
ideas are not new---they were implicit
in~\cite{cruz-martell-perezBook}.  However, we think it is worthwhile
to make them explicit here, for two reasons.  First, they will
motivate the technical details in our proofs, particularly
Theorem~\ref{thm:limited-var}.  Second, they will be helpful
to others attempting to prove extrapolation theorems in different
settings.  Finally, in Section~\ref{section:extrapol-proof} we prove
our extrapolation theorems.  By following the schema outlined in the
previous section, we actually prove more general theorems
which yield our main results as special cases.

\section{Main Theorems}
\label{section:main-theorems}

We begin with some definitions related to the variable Lebesgue
spaces.  Throughout we will follow the conventions established
in~\cite{cruz-fiorenza-book}.
Let $\Pp=\Pp(\R^n)$ be the collection of all measurable functions
$\pp : \R^n \rightarrow [1, \infty]$.   Given a set $E\subset \R^n$,
we define
\[ p_-(E) = \essinf_{x\in E} p(x), \qquad p_+(E) = \esssup_{x\in E}
p(x). \]
If $E=\R^n$, then for brevity we write $p_-$ and $p_+$.   Given $\pp$,
the conjugate exponent $\cpp$ is defined pointwise
\[ \frac{1}{p(x)}+\frac{1}{p'(x)} = 1, \]
with the convention that $1/\infty = 0$.

For our results we need to impose some regularity on the exponent
functions $\pp$.  The most important condition, one widely used in
the study of variable Lebesgue spaces, is log-H\"older
continuity.  Given $\pp\in \Pp$, we say $\pp\in LH_0$ if
there exists a constant $C_0$ such that
 \begin{equation}\label{def:LH0}
	|p(x) - p(y)| \leq \frac{C_0}{-\log(|x - y|)}, \qquad x,\, y, \in \R^n, \qquad |x - y| < 1/2,
 \end{equation}
and $\pp\in LH_\infty$ if there exists $p_{\infty}$ and $C_\infty > 0$, such that
 \begin{equation}\label{def:LHinf}
 |p(x) - p_{\infty}| \leq \frac{C_\infty}{\log(e + |x|)}, \qquad x \in \R^n.
\end{equation}	
If $\pp$ satisfies both of these conditions we write $\pp\in LH$.   It
is immediate that if $\pp \in LH$, then $\cpp \in LH$. A key consequence of log-H\"older continuity is the fact that if $1<p_-$ and
$\pp\in LH$, then the maximal operator is bounded on $\Lp$.

\begin{theorem} \label{thm:max-op}
Given $\pp \in \Pp$, suppose $1<p_-\leq p_+<\infty$ and $\pp\in LH$.
Then $\|Mf\|_\pp \leq C\|f\|_\pp$.
\end{theorem}

 However,
this condition is not necessary, and there exist exponents $\pp$ which
are not log-H\"older continuous but for which the maximal operator is
still bounded on $\Lp$.
(See~\cite{cruz-fiorenza-book,diening-harjulehto-hasto-ruzicka2010}
for further details.)

\medskip

Given a weight $w$ (again, a non-negative, measurable function) and
$\pp\in \Pp$, define the weighted variable Lebesgue space $\Lp(w)$ to
be the set of all measurable functions $f$ such that $fw \in \Lp$, and
we write $\| f \|_{L^{\pp}(w)} = \| fw \|_{\pp}$.   We
say that an operator $T$ is bounded on $\Lp(w)$ if $\|(Tf)w\|_\pp \leq
C\|fw\|_\pp$ for all $f\in \Lp(w)$.    We are interested in
weights in $A_\pp$; we restate their definition here.

\begin{definition} \label{defn:Apvar}
Given an exponent $\pp\in \Pp$ and a weight $w$ such that
$0<w(x)<\infty$ almost everywhere, we say that $w\in
A_\pp$ if
\[ [w]_{A_\pp}= \sup_B |B|^{-1} \|w\chi_B\|_\pp \|w^{-1}\chi_B\|_\cpp < \infty, \]
where the supremum is taken over all balls $B\subset \R^n$.
\end{definition}

\begin{remark}
Definition \ref{defn:Apvar} has two immediate consequences.  First, if
$w\in A_\pp$, then $w\in L^\pp_{loc}$ and $w^{-1}\in L^\cpp_{loc}$.
Second, if $w \in A_{\pp}$, then $w^{-1} \in A_{\cpp}$.
\end{remark}

For our results we will need to assume that the maximal operator is
bounded on weighted variable Lebesgue spaces.  The following result is
from~\cite{MR2927495}.

\begin{theorem} \label{thm:wtd-max}
Given $\pp\in \Pp$, $1<p_-\leq p_+<\infty$, suppose $\pp \in LH$.
Then for every $w\in A_\pp$,
\begin{equation} \label{eqn:wtd-max}
 \|Mf\|_{\Lp(w)} \leq C\|f\|_{\Lp(w)}.
\end{equation}
Conversely, given any $\pp$ and $w$, if \eqref{eqn:wtd-max}
holds for $f\in \Lp(w)$, then $p_->1$ and $w\in A_\pp$.
\end{theorem}

For the majority of our extrapolation results we prefer to state the
regularity of $\pp$ and $w$ in terms of the boundedness of the maximal
operator.  Therefore, given $\pp\in \Pp$ and a weight $w$, we will say
$(\pp, w)$ is an $M$-pair if the maximal operator is bounded on
$L^{\pp}(w)$ and $L^{\cpp}(w^{-1})$. By Theorem~\ref{thm:wtd-max} we
necessarily have $w \in A_{\pp}$ (equivalently, $w^{-1} \in A_{\cpp}$)
and $p_- > 1$.  Conversely, if $\pp\in LH$, with $p_- > 1$, then for
any $w\in A_\pp$, $(\pp, w)$ is an $M$-pair.

\begin{remark} \label{remark:duality}
 By a very deep result of
Diening~\cite{MR2166733,diening-harjulehto-hasto-ruzicka2010}, if
$1<p_-\leq p_+<\infty$, $M$ is bounded on $\Lp$ if and only if it is
bounded on $L^\cpp$.
We conjecture that the same ``duality'' result holds in the weighted
Lebesgue spaces, that is, it suffices to define an $M$-pair only by the boundedness of $M$ on $L^{\pp}(w)$.    We also conjecture
(see~\cite{MR2927495, diening-hastoPreprint2010} that if $M$ is
bounded on $\Lp$ and $w\in A_\pp$, then $M$ is bounded on $L^{\pp}(w)$.
If these two conjectures are true, then the hypotheses of our
results below would be simpler.
\end{remark}

\medskip

Though our goal is to use extrapolation to prove specific operators are
bounded on $\Lp(w)$, we will state our results more abstractly.  Following the
approach established in~\cite{MR2210118} (see
also~\cite{cruz-fiorenza-book,cruz-martell-perezBook}) we will write
our extrapolation theorems for pairs of functions $(f,g)$ contained in
some family $\F$.  Hereafter, if we write
\[ \|f\|_{X} \leq C\|g\|_Y,  \qquad (f,g) \in \F, \]
where $X$ and $Y$ are Banach function spaces (e.g., weighted classical or
variable Lebesgue spaces), then we mean that this inequality is
true for every pair $(f,g)\in \F$ such that the left-hand side of this
inequality is finite.    We will make the utility of this formulation
clear in Section~\ref{section:applications}.

\medskip

We can now state our main results. The first is a direct
generalization of the classical Rubio de Francia extrapolation theorem
and an extension of~\cite[Theorem~1.3]{MR2210118} to weighted variable
Lebesgue spaces.

\begin{theorem}\label{thm:diag-weightedvar}
Suppose that for some $p_0$, $1 < p_0 < \infty$, and every $w_0 \in A_{p_0}$,
\begin{equation}\label{hyp:diag-weightedvar}
	\int_{\R^n} f(x)^{p_0} w_0(x) dx \leq C \int_{\R^n} g(x)^{p_0}
        w_0(x) dx,
\qquad (f, g) \in \mathcal{F}.
\end{equation}
Then for any $M$-pair $(\pp, w)$,
	\begin{equation}\label{result:diag-weightedvar}
	\| f \|_{\Lp(w)} \leq C \| g  \|_{\Lp(w)}, \qquad (f, g) \in \mathcal{F}.
	\end{equation}
The theorem holds if $p_0=1$ if we assume only that the maximal
operator is bounded on $L^{\cpp}(w^{-1})$.
\end{theorem}

\begin{remark}
When $p_0=1$, Theorem~\ref{thm:diag-weightedvar} is still true and is a
special case of Theorem~\ref{thm:A1var} below.
\end{remark}

\medskip

Our second result yields off-diagonal inequalities between two
different weighted variable Lebesgue spaces.  In the constant exponent case
this result was first proved in~\cite{harboure-macias-segovia88}, and
it was proved in unweighted $\Lp$ spaces in~\cite[Theorem~1.8]{MR2210118}.  To
state it, we first define the appropriate weight classes that
generalize the  $A_p$ weights.  In the classical case
these weights were introduced in~\cite{muckenhoupt-wheeden74}.

\begin{definition} \label{defn:pq-weights}
Given $1<p\leq q<\infty$,
we say that $w\in A_{p,q}$ if
\[ \sup_B \left( \frac{1}{|B|} \int_B w(x)^q dx \right)^{1/q} \left(
  \frac{1}{|B|} \int_B w(x)^{-p'} dx \right)^{1/p'} < \infty, \]
where the supremum is taken over all balls $B\subset \R^n$.    If
$p=1$, then $w\in A_{1,q}$ if
\[ \sup_B \frac{ \avgint w(x)^q \,dx}{\essinf_{x\in B} w(x)^q} <
\infty.
\]
\end{definition}

\begin{definition} \label{defn:pq-var-weights}
Let $\pp,\,\qq \in \Pp$ be such that for some $\gamma$,
$0<\gamma<1$,
\[ \frac{1}{p(x)}-\frac{1}{q(x)} = \gamma.  \]
Given $w$ such that $0<w(x)<\infty$ almost everywhere,
 we
say that $w\in A_{\pp,\qq}$ if
\[ \sup_B |B|^{\gamma-1}\|w\chi_B\|_\qq \|w^{-1}\chi_B\|_\cpp < \infty, \]
where the supremum is taken over all balls $B\subset \R^n$.
\end{definition}

\begin{theorem}\label{thm:off-weightedvar} Suppose that for some $p_0, q_0$, $1 < p_0 \leq q_0 < \infty$, and every $w_0 \in A_{p_0, q_0}$,
 \begin{equation}\label{hyp:off-weightedvar}
	\left( \int_{\R^n} f(x)^{q_0} w_0(x)^{q_0} dx \right)^{1/q_0}
        \leq C \left( \int_{\R^n} g(x)^{p_0} w_0 (x)^{p_0} dx
        \right)^{1/p_0}, \qquad (f,g)\in \F.
	\end{equation}
Given $\pp, \qq \in \mathcal{P}$, suppose
\[ \frac{1}{p(x)} - \frac{1}{q(x)} = \frac{1}{p_0}-\frac{1}{q_0}. \]
Define $\sigma\geq 1$ by $1/\sigma'=1/p_0-1/q_0$.
If $w\in A_{\pp,\qq}$ and $(\qq/\sigma, w^\sigma)$ is an $M$-pair, then
\begin{equation}\label{result:off-weightedvar}
	\| f \|_{L^{\qq} (w)} \leq C \| g \|_{L^{\pp}(w)}, \qquad
        (f,g)\in F.
\end{equation}
The theorem holds if $p_0=1$ if we assume only that the maximal
operator is bounded on $L^{(\qq/q_0)'}(w^{-q_0})$.
\end{theorem}

\begin{remark}
When $\sigma=1$, Theorem~\ref{thm:off-weightedvar} reduces to
Theorem~\ref{thm:diag-weightedvar}.  Therefore, in proving it we will
assume that $\sigma>1$.
\end{remark}

\medskip

Our third result extends the theory of limited range extrapolation to the
weighted variable Lebesgue spaces.  This concept was introduced by
Auscher and Martell~\cite{auscher-martell07} and independently by
Duoandikoetxea {\em et
  al.}~\cite{duoandikoetxea-moyua-oruetxebarria-seijoP} in a somewhat
different form.  We generalize both their results.  To state our main result
we recall a definition: we say $w\in RH_s$ for some $s>1$ if
\[ [w]_{RH_s} = \sup_B \frac{\left(\avgint_B w(x)^s\,dx\right)^{1/s}}
{\avgint w(x)\,dx} < \infty. \]
Given a weight $w$, $w\in A_p$ for some $p\geq 1$ if and only if there
there exists $s>1$ such that
$w\in RH_s$ (see~\cite{duoandikoetxea01}).  As given
in~\cite{auscher-martell07}, limited range extrapolation in the
constant exponent case is the following.

\begin{theorem} \label{thm:limited-const}
  Given $1 < q_- < q_+ < \infty$, suppose there exists $p_0$, $q_-
  <p_0< q_+$ such that for every $w_0 \in A_{p_0/q_-} \cap
  RH_{(q_+/p_0)'}$,
\begin{equation}\label{hyp:limited-var}
	\int f(x)^{p_0} w_0 (x) dx \leq c \int g(x)^{p_0} w_0(x) dx,
        \qquad (f,g)\in \F.
\end{equation}
Then for every $p$, $q_-<p<q_+$ and every $w \in A_{p/q_-} \cap
      RH_{(q_+/p)'}$,
\[ 	\int f(x)^{p} w (x) dx \leq c \int g(x)^{p} w(x) dx,
        \qquad (f,g)\in \F. \]
\end{theorem}

In the variable exponent case we have a very different result, which
does not reduce to the constant case, Theorem~\ref{thm:limited-const}.

\begin{theorem}\label{thm:limited-var}
Given $1 < q_- < q_+ < \infty$, suppose there exists $p_0$, $q_-
  <p_0< q_+$ such that for every $w_0 \in A_{p_0/q_-} \cap
  RH_{(q_+/p_0)'}$, \eqref{hyp:limited-var} holds.
Then for every $\pp \in LH$ with $q_- < p_- \leq p_+ < q_+$,
            \begin{equation}\label{result:limited-var1}
	           \| f \|_{\pp} \leq C\| g \|_{\pp},
                \qquad (f,g)\in \F.
            \end{equation}
            More generally, there exists $p_*$, $q_-<p_*<q_+$ such
            that if let $\sigma = \frac{p_* q_-}{p_* - q_-}$, then
            there exists a constant $c = c(p_-, p_+, q_-, q_+, p_*) \in (0, 1)$, so that for every
            weight $w$ with $w^{\sigma} \in A_{\frac{\pp}{c\sigma}}$,
            we have
    \begin{equation}\label{result:limited-var2}
    \| fw \|_{\pp} \leq C \| gw \|_{\pp}.
    \end{equation}
    \end{theorem}

    \begin{remark}\label{remark:limited-constant} The two inequalities
      \ref{result:limited-var1} and \ref{result:limited-var2} follow
      from two special cases of a more general version of the theorem
      in Proposition~\ref{prop-limited}. However, the constant exponent
      result in Theorem~\ref{thm:limited-const} is from a third
      special case, and this reduction is not immediately obvious: see
      Remark \ref{remark:lim-reduction} for details. We discuss the
      relationship between these cases in
      Remark~\ref{remark:limited-cases}.
      \end{remark}

\begin{remark}  A weaker version of the unweighted
  inequality~\eqref{result:limited-var1} in
  Theorem~\ref{thm:limited-var} was implicit in Fiorenza
  {\em et al.}~\cite{MR3057132}.
\end{remark}

\begin{remark} The regularity assumption on $\pp$ in
  Theorem~\ref{thm:limited-var} can be weakened. For example, it
  follows from the proof of~\eqref{result:limited-var1} that there
  exists $s=s(q_-,q_+,p_-,p_+)$ such that it suffices to assume that
  $M$ is bounded on $L^{\pp}$ and $L^{(\pp/s)'}$.  By the duality
  property of the maximal operator (see Remark~\ref{remark:duality})
  the second assumption is equivalent to assuming $M$ is bounded on
  $L^{\pp/s}$.  Depending on whether $s>1$ or $s<1$, one of these
  assumptions implies the other, since if $M$ is bounded on $\Lp$, it
  is bounded on $L^{r\pp}$ for all $r>1$
  (\cite[Theorem~3.38]{cruz-fiorenza-book}). Regarding the constants
  in the conclusion: $c$ depends on $p_*$, and as we will see from
  the proof, the existence of
  $p_*$ is guaranteed if we take it sufficiently close to $q_-$.
\end{remark}

\begin{remark} \label{remark:power-wt}
The hypotheses on the weight $w$ for
inequality~\eqref{result:limited-var2} to hold is restrictive, but
there exist weights that satisfy them.   We have
shown that if $\pp
\in LH$ and $0\leq a < n/p_+$, then $w(x)=|x|^{-a}\in A_\pp$. (This
result will appear in~\cite{dcu-lw15}.)
Hence, if $0\leq a < cn/p_+$, $|x|^{-a\sigma}\in
  A_{\frac{\pp}{c\sigma}}$.   This result can also be used to
      construct non-trivial examples of weights that satisfy the
      hypotheses of our other results.
    \end{remark}

\medskip

We can also generalize the version of limited range extrapolation 
from~\cite{duoandikoetxea-moyua-oruetxebarria-seijoP}.

\begin{corollary}\label{cor:limited-corollary}
Given $\delta$, $0 < \delta \leq 1$, suppose that for every $w \in A_2$,
 \begin{equation}\label{hyp:limited-corollary}
	\int f(x) ^2 w(x)^{\delta} dx \leq c \int g(x)^2 w(x)^{\delta}
        dx, \qquad (f,g)\in \F.
 \end{equation}
 Then for every $\pp \in LH$ such that
 \begin{equation}\label{cor:limited-pp}
	\frac{2}{1 + \delta} < p_- \leq p_+ < \frac{2}{1 - \delta},
      \end{equation}
we have that
 \begin{equation}\label{result:limited-corollary1}
	\| f \|_{\pp} \leq C \| g \|_{\pp}, \qquad (f,g)\in \F.
 \end{equation}
More generally, for such a $\pp$, then with the same $\sigma$ in the previous theorem, there exists a constant $c \in (0, 1)$, so that for every weight $w$ such that $w^{\sigma} \in A_{\frac{\pp}{c\sigma}}$, we have
    \begin{equation}\label{result:limited-corollary2}
    \| fw \|_{\pp} \leq C \| gw \|_{\pp}.
    \end{equation}
\end{corollary}

\begin{remark}
  For simplicity we have stated Corollary~\ref{cor:limited-corollary}
  only assuming a weighted $L^2$ estimate.  A more general result is
  possible:    see~\cite[Remark~3.39]{cruz-martell-perezBook}.  An unweighted version of
  Corollary~\eqref{cor:limited-pp} that includes this generalization
  has recently been proved by Gogatishvili and Kopaliani~\cite{GK14}.
\end{remark}

\medskip

Finally, we give two variants of classical extrapolation.
We first consider extrapolation from $A_1$ weights.  This result is a
generalization of the original extrapolation theorem for variable
Lebesgue spaces in~\cite[Theorem~1.3]{MR2210118}.  It shows that we
can weaken the hypotheses of Theorem~\ref{thm:diag-weightedvar} when
$p_0=1$ and also prove results for exponents function such that
$p_-\leq 1$.  To state our result we introduce a
more general class of exponents:  we say $\pp \in \Pp_0$ if $\pp :
\R^n \rightarrow (0,\infty)$.    For such $\pp$ we define the ``norm'' $\|\cdot\|_{\pp}$
(actually a quasi-norm:  see~\cite{DCU-dw-P2014})
exactly as we do for $\pp\in \Pp$.

\begin{theorem} \label{thm:A1var}
Suppose that for some $p_0 >0 $ and every $w_0 \in A_1$,
\begin{equation} \label{eqn:A1hyp}
 \int_\subRn f(x)^{p_0} w_0 (x) dx \leq C\int_\subRn g(x)^{p_0} w_0(x)
dx, \qquad (f,g) \in \F.
\end{equation}
Given $\pp\in \Pp_0$ such that $p_- \geq p_0$, suppose that
 $w\in A_{\pp/p_0}$ and $M$ is bounded on $L^{(\pp/p_0)'}(w^{-p_0})$.
 Then
 \[ \| f \|_{\Lp(w)} \leq C\| g \|_{\Lp(w)}, \qquad (f,g) \in \F.\]
\end{theorem}

\begin{remark}
There is an important difference between Theorem~\ref{thm:A1var} (and~\cite[Theorem~1.3]{MR2210118}) and
Theorem~\ref{thm:diag-weightedvar}.  With the latter we can
extrapolate both ``up'' and ``down'': i.e., we can get results for
$\pp$ irrespective of whether $p_-$ is larger or smaller
than $p_0$.  With $A_1$ extrapolation, however, we have the restriction that
$p_-\geq p_0$.   The same situation holds in the constant exponent
case and is to be expected, since the $A_1$ case often governs ``endpoint''
inequalities.  This weaker conclusion is balanced by the weaker
hypothesis: we do not require $(\pp/p_0,
  w^{p_0})$ to be an $M$-pair,  since in the proof we will only need the
  ``dual'' inequality for the maximal operator.
\end{remark}

\begin{remark}
  The hypotheses of Theorem~\ref{thm:A1var} are
  redundant, since if $M$ is bounded on $L^{(\pp/p_0)'}(w^{-p_0})$, then
  $w^{-p_0} \in A_{(\pp/p_0)'}$, which in turn implies that $w^{p_0}\in
  A_{\pp/p_0}$.  Conversely, if we take $\pp \in LH$, then it is enough
  to assume $w^{p_0}\in A_{\pp/p_0}$.
\end{remark}

\medskip

Extrapolation can also be applied to inequalities governed by the larger class $A_\infty = \bigcup_{p>1} A_p$. The following result was first proved in~\cite{cruz-uribe-martell-perez04}.

\begin{theorem} \label{thm:Ainfty-extrapol}
If for some $p_0>0$ and every $w_0\in A_\infty$,
\begin{equation} \label{eqn:Ainfty-hyp}
 \int_\subRn f(x)^{p_0} w_0(x)\,dx \leq C
\int_\subRn g(x)^{p_0} w_0(x)\,dx, \qquad (f,g) \in \F,
\end{equation}
then the same inequality holds with $p_0$ replaced by any $p$,
$0<p<\infty$.
\end{theorem}

$A_\infty$ extrapolation in variable  Lebesgue spaces
has the following form.

\begin{theorem} \label{thm:Ainfty-extrapolvar}
Suppose that for some $p_0>0$ and every $w_0\in A_\infty$, inequality
\eqref{eqn:Ainfty-hyp} holds.   Then given $\pp \in \Pp_0$, suppose
there exists $s\leq p_-$ such that $w^s \in A_{\pp/s}$ and $M$ is bounded
on $L^{(\pp/s)'}(w^{-s})$.  Then
\[ \|f\|_{\Lp(w)} \leq C\|g\|_{\Lp(w)}, \qquad (f,g) \in \F. \]
 \end{theorem}

\begin{remark}
There is a close connection between $A_1$ and $A_\infty$
extrapolation:  see~\cite[Section~3.3]{cruz-martell-perezBook}.  We
will exploit this fact in our proof.
\end{remark}

To make the connection between Theorems~\ref{thm:Ainfty-extrapol}
and~\ref{thm:Ainfty-extrapolvar} clearer, we introduce the
notation $A_\pp^{var}$ for the weights that satisfy the variable
exponent Muckenhoupt condition.  Then if $\pp=p$ is a constant, the
hypothesis in Theorem~\ref{thm:Ainfty-extrapolvar} is $w^s \in
A_{p/s}^{var}$.  It follows at once from Definition~\ref{defn:Apvar}
that this is equivalent to $w^p \in A_{p/s}\subset A_\infty$.
Conversely, the hypothesis in Theorem~\ref{thm:Ainfty-extrapol} is
that $w^p\in A_\infty$, i.e., for some $t>1$, $w^p \in
A_t$.  Fix $s<p$ such that $t=p/s$;  then $w^p\in A_{p/s}$, or
equivalently, $w^s\in A_{p/s}^{var}$.

As the next proposition shows, the hypotheses of
Theorem~\ref{thm:Ainfty-extrapolvar} are weaker than those of
Theorem~\ref{thm:diag-weightedvar}.

\begin{prop} \label{prop:Ainfty-weaker}
Given $\pp \in \Pp$, suppose $w\in A_\pp$.  Then for every $s$,
$0<s<1$, $w^s\in A_{\pp/s}$.
\end{prop}


\section{Norm Inequalities for Operators}
\label{section:applications}

In this section we use extrapolation to prove norm
inequalities for a variety of operators on the weighted variable Lebesgue
spaces.  We will first discuss how to prove that an operator $T$ is
bounded on $\Lp(w)$ using Theorem~\ref{thm:diag-weightedvar}.  These
same ideas can be used to apply the other theorems and the details are
left to the reader.   Following this, we will give applications to some
specific operators.  Our goal is not to be exhaustive, but rather to
illustrate the utility of extrapolation by concentrating on some key
examples.  For additional applications, see~\cite{
  cruz-fiorenza-book,MR2210118, cruz-martell-perezBook}.

\subsection*{Applying extrapolation}
The key to applying Theorem~\ref{thm:diag-weightedvar} is to construct
the appropriate family $\F$.  This generally requires an approximation
argument since we need pairs $(f,g)$ such that $f$ lies in both the
appropriate weighted space to apply the hypothesis and in the target
weighted variable Lebesgue space.  The dense subsets of $L^p(w)$ are
well-known: e.g., smooth functions and bounded functions of compact
support.  These sets are also dense in $\Lp(w)$.

\begin{lemma} \label{lemma:density}
Given $\pp \in \Pp$ with $p_+<\infty$, and a weight $w\in \Lp_{loc}$,
then $L^\infty_c$, bounded functions of compact support, and
$C_c^\infty$, smooth functions of compact support,  are dense in $\Lp(w)$.
\end{lemma}

\begin{proof}
 We first prove that $L_c^\infty$is dense. The proof is
essentially the same as in the unweighted case
\cite[Theorem~2.72]{cruz-fiorenza-book}; for the convenience of the
reader we sketch the details.
  Given $f\in \Lp(w)$, define $f_n =
  \sgn(n)\min(|f(x)|,n)\chi_{B(0,n)}$.  Then $f_n\rightarrow f$
  pointwise as $n\rightarrow \infty$, and $|f_n|w \leq |f|w$.  Since
  $p_+<\infty$, we can apply the dominated convergence
  theorem~\cite[Theorem~2.62]{cruz-fiorenza-book} to conclude that
  $f_nw\rightarrow fw$ in $\Lp$; equivalently, $f_n\rightarrow f$ in
  $\Lp(w)$.

  The density of $C_c^\infty$ now follows form this.  By Lusin's
  theorem, given $f\in L_c^\infty$, for every $\epsilon>0$ there
  exists a continuous function of compact support $g_\epsilon$ such
  that $\|g\|_\infty \leq \|f\|_\infty$ and
$|D_\epsilon| = |\{ x : g(x) \neq f(x) \} | < \epsilon$.
But then
\[ \|f-g_\epsilon\|_{\Lp(w)} \leq 2\|f\|_\infty
\|\chi_{D_\epsilon}w\|_\pp. \]
Since $w\in\Lp_{loc}$, again by the dominated convergence theorem in $\Lp$, the
righthand term tends to $0$ as $\epsilon\rightarrow 0$.  Hence,
continuous functions of compact support are dense in $\Lp(w)$.  Since
every continuous function of compact support can be approximated
uniformly by smooth functions, we also have $C_c^\infty$ is dense.
\end{proof}

\medskip

Now suppose that for every $w_0 \in A_{p_0}$ and $f\in L^{p_0}(w)$, an operator $T$ satisfies
\begin{equation} \label{eqn:starting-pt}
 \int_\subRn |Tf(x)|^{p_0}w_0(x)\,dx \leq C\int_\subRn |f(x)|^{p_0} w_0(x)\,dx.
\end{equation}
 We want to show that given an $M$-pair $(\pp, w)$, $T$
is bounded on $\Lp(w)$.  Since $w\in
\Lp_{loc}$, by a standard argument
(cf.~\cite[Theorem~5.39]{cruz-fiorenza-book}) it will suffice to show
that $\|(Tf)w\|_\pp \leq C\|fw\|_\pp$ for all $f\in L^\infty_c$.
Intuitively, we want to define the family $\F$ by
\[ \F = \{ (|Tf|,|f|) : f \in L^\infty_c \}.  \]
However, we do not know {\em a priori} that $Tf \in \Lp(w)$.  To
overcome this we make a second approximation and define $\langle
Tf\rangle_n = \min(|Tf|,n)\chi_{B(0,n)}$.  Again since $ w\in
\Lp_{loc}$, we have that $\langle Tf\rangle_n\in \Lp(w)$.  Furthermore, it
is immediate that \eqref{eqn:starting-pt} holds with $|Tf|$ replaced
by $\langle Tf\rangle_n$.  Therefore, if we define
\[  \F = \{ ( \langle Tf \rangle_n,|f|) : f \in L^\infty_c, n\geq 1 \},  \]
then we can apply Theorem~\ref{thm:diag-weightedvar} and Fatou's lemma
in the variable Lebesgue spaces
(\cite[Theorem~2.61]{cruz-fiorenza-book}) to conclude that for all
$f\in L^\infty_c$,
\[ \|(Tf)w\|_\pp \leq \liminf_{n\rightarrow \infty} \|\langle Tf\rangle_n w \|_\pp
\leq C\|fw\|_\pp. \]
Similar arguments hold if we need to take $f\in C_c^\infty$ or in some
other dense set.

\subsection*{The Hardy-Littlewood maximal operator}
Although we must assume the boundedness of the maximal operator to
apply extrapolation, as an immediate consequence we get vector-valued
inequalities for it.  It is well-known that for all $p,\,q$, $1< p,\,q < \infty$, and all $w\in A_p$,
\[ \bigg\| \bigg( \sum_{k=1}^\infty (Mf_k)^q \bigg)^{1/q}\bigg
\|_{L^p(w)} \leq C
\bigg\| \bigg( \sum_{k=1}^\infty |f_k|^q \bigg)^{1/q}\bigg
\|_{L^p(w)}. \]
(See, for instance,~\cite{andersen-john80}.) From this  we
immediately get the following inequality.

\begin{corollary}
Given an $M$-pair $(\pp, w)$ and $1<q<\infty$,
\[ \bigg\| \bigg( \sum_{k=1}^\infty (Mf_k)^q \bigg)^{1/q}\bigg
\|_{\Lp(w)} \leq C
\bigg\| \bigg( \sum_{k=1}^\infty |f_k|^q \bigg)^{1/q}\bigg
\|_{\Lp(w)}. \]
\end{corollary}

This result is not particular to the maximal operator:  such vector-valued inequalities are an immediate consequence of
extrapolation defined in terms of ordered pairs of functions.   This
is proved in the constant exponent case
in~\cite[Corollary~3.12]{cruz-martell-perezBook}, and the same proof
works in our more general setting.

\begin{remark}
In the
same way, though we do not discuss them here, weak-type inequalities
can be proved using extrapolation.
See~\cite[Corollary~3.11]{cruz-martell-perezBook} and
\cite[Corollary~5.33]{cruz-fiorenza-book} for details.
\end{remark}

\begin{remark}
Vector-valued inequalities for the maximal operator play an important
role in studying functions spaces in the variable exponent setting:
see, for example,~\cite{DCU-dw-P2014,MR2498558}.
\end{remark}

\subsection*{Singular integral operators}
Let $T$ be a convolution type singular integral:  $Tf= K*f$, where
$K$ is defined on
$\R^n\setminus \{0\}$ and satisfies $\hat{K} \in L^\infty$ and
\[ |K(x)| \leq \frac{C}{|x|^n}, \quad |\nabla K(x)| \leq
\frac{C}{|x|^{n+1}}, \quad x \neq 0.  \]
More generally, we can take $T$ to be a Calder\'on-Zygmund singular
integral of
the type defined by Coifman and Meyer.
Then for all $p$, $1<p<\infty$, and $w\in A_p$,
\begin{equation} \label{eqn:sio}
 \int_\subRn |Tf(x)|^p w(x)\,dx \leq C\int_\subRn |f(x)|^pw(x)\,dx.
\end{equation}
(See~\cite{duoandikoetxea01,grafakos08b}.) As an immediate consequence we get that singular integrals are bounded
on weighted Lebesgue spaces.

\begin{corollary}
Let $T$ be a Calder\'on-Zygmund singular integral operator. {Then for any $M$-pair $(\pp, w)$,}
\[ \|Tf\|_{\Lp(w)} \leq C\|f\|_{\Lp(w)}. \]
\end{corollary}

\medskip

We can also use extrapolation to prove norm inequalities for operators that are more
singular.  Given $1<r\leq \infty$, let $\Omega \in
L^r(S^{n-1})$ satisfy $\int_{S^{n-1}} \Omega(y)\,d\sigma(y) = 0$,
where $S^{n-1}$ is the unit sphere and $\sigma$ is surface measure on
$S^{n-1}$.  Given the kernel $K$
\[ K(x) = \frac{\Omega(x/|x|)}{|x|^n}, \]
define $T_\Omega f=K*f$.  Then for all $p>r'$ and $w\in A_{p/r'}$,
\eqref{eqn:sio} holds for $T_\Omega$~\cite{duoandikoetxea93,watson90}.
This is a limiting case of Theorem~\ref{thm:limited-var}, with
$q_-=r'$ and $q_+=\infty$.  However, it is more straightforward to
apply Theorem~\ref{thm:diag-weightedvar} rescaling.  If we rewrite
\eqref{eqn:sio} as
\begin{equation} \label{eqn:sio-r}
  \int_\subRn \big(|T_\Omega f(x)|^{r'}\big)^{p/r'} w(x)\,dx
\leq C\int_\subRn \big(|f(x)|^{r'}\big)^{p/r'}w(x)\,dx,
\end{equation}
then for any $M$-pair $(\pp, w)$,
\[ \||T_\Omega f|^{r'}w\|_{\Lp} \leq C \||f|^{r'}w\|_{\Lp}. \]
In particular, if we replace $w$ by $w^{r'}$ and $\pp$ by $\pp/r'$,
then by dilation we get a variable exponent analog of
inequality~\eqref{eqn:sio-r}.

\begin{corollary}
Given $\pp$ and $w$ such that {$(\pp/r', w^{r'})$ is an $M$-pair, we have}
\[ \|T_\Omega f\|_{\Lp(w)} \leq C\|f\|_{\Lp(w)}. \]
\end{corollary}

\subsection*{Off-diagonal operators}
Given $\alpha$, $0<\alpha<n$, the fractional integral operator of
order $\alpha$ (also referred to as the Riesz potential) is the
positive integral operator
\[ I_\alpha f(x) = \int_\subRn \frac{f(y)}{|x-y|^{n-\alpha}}\,dy.\]
The associated fractional maximal operator $M_\alpha$ is defined by
\[ M_\alpha f(x) = \sup_{B} |B|^{\alpha/n} \avgint_B |f(y)|\,dy \cdot
\chi_B(x).  \]
Weighted inequalities for both of these operators are governed by the
$A_{p,q}$ weights in Definition~\ref{defn:pq-weights}:
given $p$, $1<p<n/\alpha$, and $q$ such that
$\frac{1}{p}-\frac{1}{q}=\frac{\alpha}{n}$, then for all $w\in A_{p,q}$,
\[ \left(\int_\subRn |I_\alpha f(x) w(x)|^q \,dx \right)^{1/q}
\leq C \left(\int_\subRn |f(x) w(x)|^p \,dx \right)^{1/p}; \]
the same inequality holds if $I_\alpha$ is replaced by $M_\alpha$~\cite{muckenhoupt-wheeden74}.
Therefore, we can apply Theorem~\ref{thm:off-weightedvar} (using the
obvious variant of the technical reduction discussed at the beginning
of this section) to get the following result.

\begin{corollary}
Given $\alpha$, $0<\alpha<n$, suppose exponents $\pp,\,\qq$ are  such that
$p_+<n/\alpha$ and
$\frac{1}{p(x)}-\frac{1}{q(x)}=\frac{\alpha}{n}$.    Let $\sigma=(n/\alpha)'$.
 Then for all $M$-pairs $(\qq/\sigma, w^{\sigma})$,
\begin{gather*}
\|(I_\alpha f)\|_{\Lq(w)} \leq C\|f\|_{\Lp(w)}, \\
\|(M_\alpha f)\|_{\Lq(w)} \leq C\|f\|_{\Lp(w)}.
\end{gather*}
\end{corollary}

\begin{remark}
The restriction $p_+<n/\alpha$ is natural for the fractional integral
operator, since in the constant exponent case $I_\alpha$ does not map
$L^{n/\alpha}$ to $L^\infty$.  On the other hand, $M_\alpha$ does;
moreover, in the unweighted case, if $p_+=n/\alpha$, then $\|M_\alpha
f\|_\qq \leq C\|f\|_\pp$.  (See~\cite{MR2493649, cruz-fiorenza-book}.)
Therefore, we conjecture that the same is true in the weighted case;
this question is still open even for $\alpha=0$ and $p_+=\infty$.
\end{remark}

\subsection*{Coifman-Fefferman type inequalities}

There are a variety of norm inequalities that compare two operators,
usually of the form
\[ \int_\subRn |Tf(x)|^p w(x)\,dx \leq C\int_\subRn |Sf(x)|^p
w(x)\,dx,  \]
where $w\in A_\infty$.  The first such inequality, due to Coifman and
Fefferman~\cite{coifman-fefferman74}, compared singular integrals and
the Hardy-Littlewood maximal operator, and there have been a number of results proved since:  see~\cite[Chapter~9]{cruz-martell-perezBook}.    We can
use Theorem~\ref{thm:Ainfty-extrapol} to extend such inequalities to
the weighted variable Lebesgue spaces.

We illustrate this by considering one such inequality in particular,
the Fefferman-Stein inequality for the sharp maximal operator.
(See~\cite{duoandikoetxea01}.)   Recall that the  the sharp maximal
function is defined by
\[ M^\# f(x) = \sup_B \avgint_B |f(y)-f_B|\,dy \cdot \chi_B(x), \]
where $f_B = \avgint_B f(x)\,dx$.  Though pointwise smaller than the
maximal operator, we have that for all $p$,
$0<p<\infty$ and $w\in A_\infty$,
\begin{equation*}
 \int_\subRn Mf(x)^p w(x)\,dx \leq C\int_\subRn M^\#f(x)^p w(x)\,dx.
\end{equation*}
Then by
Theorem~\ref{thm:Ainfty-extrapolvar} we immediately get the
following.

\begin{corollary}
Given $\pp\in \Pp$ and a weight $w$, suppose there exists $s<p_-$ such
that $w^{s} \in A_{\pp/s}$ and $M$ is bounded on $L^{(\pp/s)'}(w^{-s})$.  Then
\[ \|f\|_{\Lp(w)} \leq C\|M^\# f\|_{\Lp(w)}. \]
\end{corollary}

In exactly the same way other Coifman-Fefferman type inequalities can
be extended to the variable Lebesgue space setting.

\subsection*{Operators with a restricted range of exponents}
Certain types of operators are not bounded on $L^p$ for every $p$,
$1<p<\infty$, but only for $p$ in some interval, say $q_-<p<q_+$.  In
this case it is natural to conjecture that such operators are bounded
on $\Lp$ provided that $q_-<p_-\leq p_+<q_+$, and that weighted
inequalities hold in the same range for suitable weights $w$.  Here we
consider two operators: the spherical maximal operator and the
Riesz transforms associated with certain elliptic operators.

The spherical maximal operator is defined by
\[ \M f(x) = \sup_{t>0} \left|\int_{S^{n-1}} f(x-ty)d\sigma(y)\right|,\]
where $S^{n-1}$ is the unit sphere in $\R^n$ and $d\sigma$ is surface
measure on the sphere.
  Stein~\cite{MR0420116} proved that for $n\geq 3$, $\M$ is bounded on $L^p$ if and
  only if $p>\frac{n}{n-1}$.  Weighted norm inequalities are true for
  the same values of $p$, but require strong conditions on the weight.
   Cowling, {\em et al.}~\cite{MR1922609} proved that
if
\[  \frac{n}{n-1}< p < \infty \quad \text{and} \quad
\max\left(0,1-\frac{p}{n}\right) \leq \delta \leq \frac{n-2}{n-1}, \]
and if
\begin{equation} \label{eqn:cowling-wt}
 w = u_1^\delta u_2^{\delta(n-1)-(n-2)},  \qquad u_1,\,u_2 \in
A_1,
\end{equation}
then $\M : L^p(w) \rightarrow L^p(w)$.

If we combine this result with Theorem~\ref{thm:limited-var} we get
the following estimates in the variable Lebesgue spaces.

\begin{corollary} \label{cor:fgk}
Fix $n\geq 3$. Given $\pp \in LH$ such that $\frac{n}{n-1}<p_-\leq p_+< (n-1)p_-$,
then
\begin{equation} \label{eqn:fgk1}
 \|\M f\|_\pp \leq C\|f\|_\pp.
\end{equation}
Moreover, if for some $\sigma>\frac{n-1}{n-2}p_-$, $w^\sigma \in
A_{\frac{\pp}{c\sigma}}$, where $c\in (0,1)$ is as in the statement of
Theorem~\ref{thm:limited-var}, then
\begin{equation} \label{eqn:fgk2}
 \|(\M f)w\|_\pp \leq C\|fw\|_\pp.
\end{equation}
\end{corollary}

\begin{proof}
To apply Theorem~\ref{thm:limited-var} we need to restate
the hypotheses of the above weighted norm inequality.
By the information encoded in the factorization of $A_p$ weights
(see~\cite[Theorems~2.1, 2.3, 5.1]{cruz-uribe-neugebauer95}), if $w$
is given by \eqref{eqn:cowling-wt}, then $w\in
A_t\cap RH_{1/\delta}$, where $1-t=\delta(n-1)-(n-2)$ or
$t=(n-1)(1-\delta)$.  Therefore, if we fix any $p_0>\frac{n}{n-1}$, we
have that $w\in A_{p_0/q_-}\cap RH_{(q_+/p_0)'}$, where
\[ q_- = \frac{p_0}{(n-1)(1-\delta)}, \qquad q_+ = \frac{p_0}{1-\delta} =
(n-1)q_-. \]
Conversely, if we take any $w\in A_{p_0/q_-}\cap RH_{(q_+/p_0)'}$,
then it can be written in the form~\eqref{eqn:cowling-wt}.

Given this reformulation we can apply Theorem~\ref{thm:limited-var}.
To prove the unweighted inequality \eqref{eqn:fgk1}, fix $\pp$ such
that $\frac{n}{n-1}<p_-\leq p_+<(n-1)p_-$.  Note that if we fix
$\delta=\frac{n-2}{n-1}$, then $q_-=p_0$, so if we take $p_0=p_-$ and take
values of $\delta$ close to $\frac{n-2}{n-1}$  we see that we can get
$q_-$ as close to $p_-$ as desired.  In particular, we can get
$p_+<(n-1)q_-=q_+$.   Inequality~\eqref{eqn:fgk1} now follows from
inequality~\eqref{result:limited-var2} in Theorem~\ref{thm:limited-var}.

\medskip

To prove the weighted inequality \eqref{eqn:fgk2}, we argue similarly.
Fix $\pp$ and $\sigma>\frac{n-1}{n-2}p_-$.  Now choose a value of
$p_0$ and fix $\,q_-,\,q_+$ as before.  Then we have that
\[ \sigma>\frac{n-1}{n-2}q_-  = \frac{(n-1)q_-^2}{(n-1)q_- - q_-}. \]
We now apply limited range extrapolation in the constant exponent
case, Theorem~\ref{thm:limited-const}; this shows that we can now take
{\em a posteriori} any value $p_0$, $q_-<p_0<q_+=(n-1)q_-$.  In
particular, we can take $p_0$ as close to $(n-1)q_-$ as we want.  Fix
$p_0$ so that
\begin{equation} \label{eqn:sigma-bound}
\sigma \geq \frac{p_0 q_-}{p_0-q_-}.
\end{equation}
By Proposition~\ref{prop:Ainfty-weaker}, if $w^\sigma \in
A_{\frac{\pp}{c\sigma}}$, then the same inclusion holds for any smaller
  value of $\sigma$, so we may assume without loss of generality that
  equality holds in \eqref{eqn:sigma-bound}.   But then we can apply
  Theorem~\ref{thm:limited-var} starting from our new value of $p_0$
  and using this value of $\sigma$ to get~\eqref{eqn:fgk2}.
\end{proof}

Inequality~\eqref{eqn:fgk1} in Corollary~\ref{cor:fgk}
  was originally proved by Fiorenza~{\em et al.}~\cite{MR3057132};
  their proof relied on a extrapolation argument which was a slightly
  weaker, unweighted version of Theorem~\ref{thm:limited-var}.

A surprising feature of this result is that while there are weighted
inequalities for any value of $p>\frac{n}{n-1}$, variable Lebesgue
space bounds only hold for exponents with bounded oscillation.   This
is not an artifact of the proof:  in ~\cite{MR3057132} they
  also proved that if the spherical maximal operator is bounded on
  $\Lp$, then $p_+\leq np_-$; it is conjectured that this bound is sharp.
  To prove this via extrapolation it suffices to show that in
  the above weighted norm inequality we could replace the upper bound
  on $\delta$ by $\frac{n-1}{n}$.    It is unclear if this is
  possible, though we note that in~\cite[p.~83]{MR1922609} they
  conjectured that one could take weights of the form
  $w=u_1^{\frac{n-1}{n}}$ which is a special case.

\bigskip

A second kind of operator that satisfies norm inequalities with a limited range
of exponents is the Riesz transform associated to complex elliptic
operators in divergence form.  We sketch the basic properties of these
operators; for complete information, see Auscher~\cite{auscher2007}.

Let $A$ be an $n\times n$, $n\geq 3$, matrix of
complex-valued measurable functions, and assume that $A$ satisfies the
ellipticity conditions
\[  \lambda | \xi | ^{2}\leq \text{Re}\langle A\xi
    ,\xi \rangle, \quad
  |\langle A\xi ,\eta \rangle |\leq \Lambda |\xi
  ||\eta |, \quad \xi ,\,\eta \in \mathbb{C}^{n}
\quad 0<\lambda < \Lambda. \]
Let $L= -\text{div}A\nabla$.    Then $L$ satisfies an $L^2$ functional
calculus, so that the square root operator $L^{1/2}$ is well defined.
The Kato conjecture asserted that this operator satisfies
\[  \|L^{1/2} f\|_2 \approx \|\nabla f\|_2, \qquad f \in W^{1,2}.  \]
This was proved by Auscher~{\em et
  al.}~\cite{auscher-hofmann-lacey-mcintosh-tchamitchian02}.  As a
consequence of this we have that the Riesz transform associated to
$L$, $\nabla L^{-1/2}$, also satisfies $L^2$ bounds:
\[ \|\nabla L^{-1/2} f \|_2 \leq C\|f\|_2. \]
This operator also satisfies weighted $L^p$ bounds for $p$ close to $2$.
Auscher and Martell~\cite{auscher-martell06} proved that
there exist constants $q_-=q_-(L)<\frac{2n}{n+2}<2$ and $q_+=q_+(L)>2$
such that if $q_-<p<q_+$ and $w\in A_{p/q_-}\cap RH_{(q_+/p)'}$, then
\[ \|\nabla L^{-1/2} f \|_{L^p(w)} \leq C\|f\|_{L^p(w)}. \]
By Theorem~\ref{thm:limited-var} we can
extend this result to the variable Lebesgue spaces.

\begin{corollary}
Given an elliptic operator $L$ as defined above, suppose the exponent $\pp \in LH$ is such that
$q_-(L)<p_-\leq p_+ <q_+(L)$.  Then
\[ \|\nabla L^{-1/2} f \|_\pp \leq C\|f\|_\pp,\]
and for any weight $w$ such that $w^n \in A_{\frac{\pp}{cn}}$, then
\[ \|(\nabla L^{-1/2} f)w \|_\pp \leq C\|fw\|_\pp.\]
\end{corollary}

\begin{proof}
The unweighted inequality is immediate.  For the weighted inequality
we take $p_0=2$, and we take a larger value for $\sigma$ (possible by
Proposition~\ref{prop:Ainfty-weaker}) by replacing $q_-$ by the upper
bound $\frac{2n}{n+2}$.   This gives $\sigma=n$.
\end{proof}

\begin{remark}
Bongioanni {\em et al.}~\cite{MR2720705} introduced a class of weights that generalize
the Muckenhoupt $A_p$ weights and are the appropriate class for
studying weighted norm inequalities for the Riesz transforms related
to Schr\"odinger operators which in many cases satisfy limited range
inequalities.  They also showed that the theory of  extrapolation
could be extended to these weight classes~\cite{MR3042701}.  It would
be of interest to determine if their results could be extended to the
appropriate scale of weighted variable Lebesgue spaces.
\end{remark}

\section{The General Approach to Extrapolation}
\label{section:general}

In this section we give a broad overview of the way in which we prove
each of our extrapolation theorems.  We have
chosen to organize the arguments in a way which does not yield the most
elegant proof but which does make clearer the process by which we
found the proof.   This discussion should be seen as a complement to
the overview of extrapolation given
in~\cite[Chapter~2]{cruz-martell-perezBook}; we believe that it will
be useful for attempts to prove extrapolation theorems in other
contexts.

All of our proofs use five basic tools:  dilation, duality, H\"older's
inequality, reverse
factorization and the Rubio de Francia algorithm.   By dilation we
mean the property that for any exponent $\pp$ and any $s>0$,
$\|f\|_\pp^s = \||f|^s\|_{\pp/s}$.
For constant exponents this is trivial, and even for general exponent
functions it is an immediate consequence of the
definition~\eqref{eqn:defn-norm}.   By duality
(see \cite[Section~2.8]{cruz-fiorenza-book}) we have that given $f\in
\Lp$, there exists $h\in L^\cpp$, $\|h\|_\cpp=1$, such that
\[ \|f\|_\pp \leq C\int_\subRn f(x)h(x)\,dx; \]
conversely, by H\"older's inequality
\cite[Section~2.4]{cruz-fiorenza-book}, if $f\in \Lp$ and $h\in
L^\cpp$, then
\[ \int_\subRn |f(x)h(x)|\,dx \leq C\|f\|_\pp\|h\|_\cpp. \]
(In both cases the constant depends only on $\pp$.)  To construct the
weight $w\in A_{p_0}$ needed to apply the hypothesis, we use reverse
factorization:  the property that if $\mu_1,\,\mu_2 \in A_1$, then
$w_0=\mu_1\mu_2^{1-p_0} \in A_{p_0}$.
(See~\cite[Prop.~7.2]{duoandikoetxea01}.)  Finally, to find the $A_1$
weights we apply the Rubio de Francia extrapolation algorithm in the
following form.

\begin{prop}\label{prop:H_j}
  Given $\rr \in \Pp$, suppose ${\mu}$ is a weight such that $M$ is
  bounded on $L^{\rr}({\mu})$. For a positive function $h \in L_{loc}^1$,
  with $Mh(x)< \infty$ almost everywhere, define:
	\[ \Rj h(x) = \sum_{k = 0}^{\infty} \frac{M^{k}h(x)}{2^k \| M
          \|_{L^\rr({\mu})}^k}. \]
Then: $(1)$ $h(x)\leq \Rj h(x)$;
$(2)$ $\|\Rj h\|_{{L^{\rr}(\mu)}} \leq 2\|h\|_{{L^{\rr}(\mu)}}$;
$(3)$ $\Rj h \in A_1$, with $[\Rj h]_{A_1} \leq 2 \| M \|_{{L^{\rr}(\mu)}}$.
More generally, for fixed constants $\alpha > 0$ and $\beta \in \R$, and another weight $w$, define the operator
	\[ H = Hh = \Rj (h^{\alpha} w^{\beta})^{1/\alpha} w^{-\beta/\alpha}. \]
Then: $(1)$  $h(x) \leq H(x)$;
$(2)$ {Let $v = w^{\beta/\alpha} \mu^{1/\alpha}$. Then $H$ is bounded on $L^{\alpha \rr}(v)$, with $\|H\|_{L^{\alpha \pp}({v})} \leq
  2\|h\|_{L^{\alpha \pp}({v})}$};
$(3)$ $H^{\alpha} w^{\beta} \in A_1$, with $[H^{\alpha}
  w^{\beta}]_{A_1} \leq 2 \| M \|_{L^{\rr}({\mu})}$.
\end{prop}

\begin{proof}
  The proof is straightforward and essentially the same as in the
  constant exponent case
  (see~\cite[Chapter~2]{cruz-martell-perezBook}): property (1) for $\Rj$
  is immediate; property (2) follows from our assumption that $M$ is
  bounded; and property (3) follows from the fact that $M$ is
  sublinear.  The properties of $H$ are immediate consequences of
  dilation and
  those for~$\Rj$.
\end{proof}

To prove our extrapolation theorems we use these tools to reduce the
quantity we want to estimate (e.g., the lefthand term in
\eqref{result:diag-weightedvar}, \eqref{result:off-weightedvar},
\eqref{result:limited-var1}, \eqref{result:limited-var2}) to something we can apply our hypothesis to (e.g., a weighted integral in the form of the lefthand side of \eqref{hyp:diag-weightedvar}, \eqref{hyp:off-weightedvar},
\eqref{hyp:limited-var}). Let us use Theorem~\ref{thm:diag-weightedvar} as an example. We first fix a weight $w$
satisfying our hypotheses and a pair $(f,g)\in \F$.  For technical
reasons we introduce a new function $h_1$ that depends on both $f$ and
$g$: intuitively, $h_1=g$, but we introduce a term involving $f$ so
that we can prove that the integral corresponding to the lefthand side
of the weighted norm inequality in the hypothesis is finite.  We also
define it to have uniformly bounded norm.  We majorize it by an
operator $H_1$ with constants $\alpha_1$ and $\beta_1$ to be
determined.  If we first apply dilation with an exponent $s>0$ and
then duality, we get a function $h_2$, also with uniformly bounded
norm, which we majorize by a second operator $H_2$ with constants
$\alpha_2$ and $\beta_2$.  We multiply and
divide by $H_1^\gamma$, $\gamma>0$, and apply H\"older's inequality to get, for
example,
\[ \|f\|_{\Lp(w)}^s \leq
\left(\int_\subRn f^{p_0} H_1^{-\gamma {(p_0/s)}} \,H_2 w^s\,dx\right)^{s/p_0}
\left(\int_\subRn H_1^{\gamma(p_0/s)'} \,H_2
  w^s\,dx\right)^{1/(p_0/s)'}. \]

Our goal is to show that the second integral is uniformly bounded, and
the first is bounded by the righthand side of our desired conclusion.
To do so we need to find appropriate values for the six undetermined parameters:  $\alpha_j$, $\beta_j$,
$1\leq j \leq 2$, $s$ and $\gamma$.   These parameters are subject to
the following constraints:
\begin{enumerate}

\item  Since we know which (unweighted) variable Lebesgue space $h_2$
  belongs to (e.g., $h_2 \in L^{(\frac{\pp}{s})'}$), we will assume that $H_2=\Rj_2 (h_2^{\alpha_2} w^{\beta_2})^{1/\alpha_2} w^{-\beta_2/\alpha_2}$
  is bounded there too. We can then use
  Proposition~\ref{prop:H_j} ``backwards'' (i.e., set $v=1$,
  $(\frac{\pp}{s})'=\alpha_2 \rr$ and solve for $\mu$) to deduce that we need the
  maximal operator $M$ bounded on $L^{(\pp/s)'/\alpha_2}(w^{-\beta_2})$.
  This gives constraints on $\alpha_2$ and $\beta_2$.

\item Similarly, we want $H_1=\Rj_1 (h_1^{\alpha_1}
  w^{\beta_1})^{1/\alpha_1} w^{-\beta_1/\alpha_1}$ to be bounded on
  the same space in which $h_1$ is contained, and again by
  Proposition~\ref{prop:H_j} (taking $v=w$ and $\pp=\alpha_1\rr$) this
  means that we need $M$ to be bounded on
  $L^{\pp/\alpha_1}(w^{\alpha_1 - \beta_1})$.  This gives constraints
  on $\alpha_1$, $\beta_1$ and $\gamma$.

\item Lastly, to apply our hypothesis, we need $H_1^{-\gamma {(p_0/s)}} \,H_2
  w^s$ to satisfy the $A_{p_0}$ condition.  To apply reverse
  factorization (since $H_1$ and $H_2$ both yield $A_1$ weights) we
  get more constraints on all the parameters (in particular
  on $s$).
\end{enumerate}
If we combine all of these constraints we are able to find sufficient
conditions on the exponent $\pp$ and the weight $w$ to get the desired conclusion.

In each of the proofs in Section~\ref{section:extrapol-proof} below,
we follow this schema.   Some of the parameters described above have
their values determined, but others are still free.   For our first
three theorems we will prove
a (seemingly) more general result, in the sense that we will show that
the desired weighted norm inequality holds for a family of weight
classes parameterized by $\beta_1$ (the constant from $H_1$) and $s$
(the constant that determines the dual space).  We will get the stated
result by choosing appropriate values for these parameters.

For Theorem~\ref{thm:diag-weightedvar} one can see the choice of the
parameters as simply what is necessary to get the result that is the
obvious analog of the classical Rubio de Francia extrapolation
theorem.  However, we will also show, in the special case of power
weights, that our choice of parameters is in some sense optimal.  The
proof of off-diagonal extrapolation,
Theorem~\ref{thm:off-weightedvar}, will follow the same pattern.
However, the proof has some technical difficulties related to the
variable Lebesgue space norm, and requires more care in choosing the
parameters.

For both Theorems~\ref{thm:diag-weightedvar}
and~\ref{thm:off-weightedvar}, the proofs would be simpler if we had
simply fixed our parameters initially, without motivating our choices.
Indeed, we admit that when we first proved each result we chose our
parameters in an {\em ad hoc} fashion, justifying our choices by the
fact that we got the desired outcome.  However, in proving limited
range extrapolation, Theorem~\ref{thm:limited-var}, we discovered that
the ``right'' parameters were not obvious: none of our initial choices
led to a meaningful result, let alone one analogous to the constant
exponent case.  Ultimately we used the approach outlined above in order to
discover what was actually going on.  We have chosen to retain it here
since it both illuminates our final result and makes clear why the
constant exponent theorem does not immediately generalize to the
variable space setting.  But then, in order to help the reader
understand our approach, we chose to write the previous two proofs in
this more general fashion.

Finally, extrapolation with $A_\infty$ and $A_1$ weights,
Theorems~\ref{thm:Ainfty-extrapolvar} and~\ref{thm:A1var}, requires
some minor modification to our general approach; we will make these clear in
the course of the proofs.


\section{Proof of Theorems}
\label{section:extrapol-proof}
In this section we give the proofs of all the results in
Section~\ref{section:main-theorems}.

\subsection*{Proof of Theorem \ref{thm:diag-weightedvar}}
When $p_0=1$, Theorem \ref{thm:diag-weightedvar} is a special case of
Theorem~\ref{thm:A1var}, so here we will assume $p_0>1$.
We will prove the following
proposition.

\begin{prop} \label{prop-diag}
Suppose \eqref{hyp:diag-weightedvar} holds for some $p_0>1$.  Fix $\pp
\in \Pp$, $\beta_1 \in \R$ and choose any  $s$ such that
    \begin{align}\label{eqn:s-diagonal}
    \max\big(0,p_0 - p_-(p_0 - 1)\big) < s < \min(p_-, p_0).
    \end{align}
    Let $\alpha_1 = \frac{p_0 - s}{p_0 - 1}$ and $\beta_2 = s - \beta_1(1 -
    p_0)$. If $M$ is bounded on $L^{\pp/\alpha_1}(w^{\alpha_1 -
      \beta_1})$ and $L^{(\pp/s)'} (w^{-\beta_2})$, then $\| fw
    \|_{\pp} \leq C \| gw \|_{\pp}$.
  \end{prop}

  The constant $s$ comes from duality and the constants $\alpha_j$ and
  $\beta_j$ are from using
  Proposition~\ref{prop:H_j} to define $H_1$ and $H_2$; the values and
  constraints are the only ones which arise in applying the method
  outlined in Section~\ref{section:general}.

To prove Theorem~\ref{thm:diag-weightedvar} it is enough to take $s=1$
and $\beta_1=0$.  Then \eqref{eqn:s-diagonal} holds (since $p_0>1$)
and the conditions on the maximal operator reduce to saying that $M$ is bounded on
$\Lp(w)$ and $L^\cpp(w^{-1})$:  that is, that $(\pp,w)$ is an
$M$-pair.  We will consider other choices of parameters in
Remark~\ref{remark:optimal} below.

\begin{proof}
Let $(f, g) \in \mathcal{F}$ with $\|f\|_{\Lp(w)}<\infty$.  Without loss
of generality we may assume $\|f\|_{\Lp(w)}>0$ and
$\|g\|_{\Lp(w)}<\infty$ since otherwise there is nothing to prove.  We
may also assume $\|g\|_{\Lp(w)}>0$: otherwise, $g(x)=0$ almost
everywhere, and so by our assumption~\eqref{hyp:diag-weightedvar}
(perhaps via an approximation argument like the one in
Section~\ref{section:applications}) we get that $f(x)=0$ a.e. Define
    \[ h_1 = \frac{f}{\| f \|_{\Lp(w)}} + \frac{g}{\| g  \|_{\Lp(w)}}. \]
Then $h_1 \in L^{\pp}(w)$ and $\|h_1\|_{\Lp(w)}\leq 2$.

We will use Proposition~\ref{prop:H_j} to define the two operators
$H_1$ and $H_2$,
\begin{equation} \label{eqn:H1-H2}
 H_1 = \Rj_1(h_1^{\alpha_1}
w^{\beta_1})^{1/\alpha_1}w^{-\beta_1/\alpha_1}, \qquad
 H_2 = \Rj_2(h_2^{\alpha_2}
w^{\beta_2})^{1/\alpha_2}w^{{-}\beta_2/\alpha_2},
\end{equation}
where $h_2$ will be  fixed momentarily.
Fix $s$, $0<s< \max(p_0,p_-)$.   By dilation,
duality and H\"older's inequality, there exists $h_2 \in
L^{(\pp/s)'}$, $\|h_2\|_{(\pp/s)'}=1$, such that for any $\gamma
> 0$,
    \begin{align} \label{eqn:first-reduction}
    \| fw \|_{\pp}^s
        &\leq C\int_\subRn f^s w^s h_2 \,dx  \leq \int_\subRn f^s H_1^{\gamma} H_1^{-\gamma} H_2 w^s \,dx \\ \notag
        &\leq C\left( \int_\subRn f^{p_0} H_1^{-\gamma (p_0/s)} H_2 w^s \,dx
        \right)^{s/p_0}
\left( \int_\subRn H_1^{\gamma (p_0/s)'} H_2 w^s \,dx
\right)^{1/(p_0/s)'}  \\ \notag
& = I_1^{s/p_0} I_2^{1/(p_0/s)'}.
    \end{align}

    We will first find assumptions that let us show that $I_2$ is uniformly bounded. Since $h_1\in
    L^{\pp}(w)$ and $h_2 \in L^{(\pp/s)'}$,  we must have that $H_1$ and
    $H_2$ are bounded on these spaces.   To get the norm of $H_2$ in
    $L^{(\pp/s)'}$ we  apply H\"older's inequality with exponent
    $\pp/s$ to get
\[ I_2 \leq C\| H_1^{\gamma (p_0/s)'} w^s \|_{\pp/s} \| H_2 \|_{(\pp/s)'}. \]
To use our assumption that $H_1$ is bounded on $L^{\pp}(w)$ we need to
fix $\gamma = \frac{s}{(p_0/s)'}$.  Then by dilation and the
properties of $H_1$ and $H_2$ in Proposition~\ref{prop:H_j} we have that
\begin{gather*}
 \| H_1^{\gamma (p_0/s)'} w^s \|_{\pp/s} =
\|H_1 w \|_{\pp}^s \leq 2^s\| h_1 w \|_{\pp}^s \leq 4^s \\
\intertext{and}
\| H_2 \|_{(\pp/s)'} \leq 2\| h_2 \|_{(\pp/s)'} = 2.
\end{gather*}
For $H_1$ and $H_2$ to be bounded on these spaces, by Proposition
\ref{prop:H_j}, we must have that the maximal operator satisfies
\begin{equation} \label{eqn:M-constraints}
M \text{ bounded on } L^{\pp/\alpha_1} (w^{\alpha_1 - \beta_1}) \text{
  and }
L^{(\pp/s)'/\alpha_2} (w^{-\beta_2}).
\end{equation}
A necessary condition for
this is that
that $p_-/\alpha_1>1$ and $[(\pp/s)'/\alpha_2]_->1$, or equivalently,
\[  p_- > \alpha_1, \qquad  (p_+/s)' > \alpha_2.  \]

We must now estimate $I_1$; with our choice of $\gamma$ it can be
written as
\[    I_1 = \int_\subRn f^{p_0} H_1^{s - p_0} H_2 w^s \,dx. \]
In order to apply \eqref{hyp:diag-weightedvar}, we  must show that $I_1$ is finite. Since $h_1 \leq H_1$, by H\"{o}lder's inequality
  \begin{multline*}
    I_1
\leq \int_\subRn  f^{p_0} \left( \frac{f}{\| fw \|_{\pp}} \right)^{{s} -
  p_0} H_2 w^{{s}} \,dx \\
=\| fw \|_{\pp}^{p_0 - {s}} \int_\subRn f^{{s}} w^{{s}} H_2 \,dx
\leq \| fw \|_{\pp}^{p_0 - {s}} \| fw \|_{\pp}^s \| H_2 \|_{{(\pp/s)'}} < \infty.
\end{multline*}
Suppose for the moment that $w_0 = H_1^{s - p_0} H_2 w^s \in A_{p_0}$;
then we can use \eqref{hyp:diag-weightedvar} to estimate $I_1$.
Again since $h_1 \leq H_1$ and by H\"{o}lder's inequality,
\begin{multline*}
  I_1
\leq C\int_\subRn g^{p_0} H_1^{{s} - p_0} H_2 w^{{s}} \,dx
 \leq C\int_\subRn g^{p_0} \left( \frac{g}{\| gw \|_{\pp}} \right)^{{s} - p_0} H_2 w^{{s}} \,dx \\
  = C\| gw \|_{\pp}^{p_0 - {s}} \int_\subRn g^{{s}} H_2 w^{{s}} \,dx
 \leq C\| gw \|_{\pp}^{p_0 - {s}} \| gw \|_{\pp}^{{s}} \| H_2 \|_{{(\pp/s)'}}
 \leq C \| gw \|_{\pp}^{p_0}.
\end{multline*}
If we combine this with the previous inequalities we get the desired
norm inequality.

To complete the proof we must determine constraints on the parameters
so that $H_1^{s - p_0} H_2 w^s \in A_{p_0}$.  By reverse
factorization and Proposition~\ref{prop:H_j} we need to fix our
parameters so that
 \[  H_1^{s - p_0} H_2 w^s = \left[ H_1^{\frac{p_0 - s}{p_0 - 1}}
   w^{\beta_1} \right]^{1 - p_0} H_2 w^{s - \beta_1 (1 - p_0)} =
\left[ H_1^{\alpha_1}w^{\beta_1} \right]^{1-p_0}
H_2^{\alpha_2}w^{\beta_2}. \]
Equating the exponents we get that
\begin{equation} \label{eqn:final-constraints}
 \alpha_1 = \frac{p_0 - s}{p_0 - 1}, \quad \beta_1 \in \R, \quad
\alpha_2 = 1, \quad \beta_2 = s - \beta_1 (1 - p_0).
\end{equation}
(In other words, there is no constraint on $\beta_1$.)  Since above we
assumed $s<p_0$, we have $\alpha_1>0$ as required in
Proposition~\ref{prop:H_j}.    Above, we required that
$p_->\alpha_1$; combining this with the new constraint we have that
$s > p_0 - p_-(p_0-1)$.  With $\alpha_2=1$, the second restriction
from above, that $(p_+/s)'>\alpha_2$, always holds.

To summarize:  we have shown that given a
constant $s$ such that~\eqref{eqn:s-diagonal} holds, and constants
$\alpha_j,\,\beta_j$ as in~\eqref{eqn:final-constraints}, and if the
maximal operator satisfies~\eqref{eqn:M-constraints}, then the desired
weighted norm inequality holds.   This completes the proof.
\end{proof}

\medskip

\begin{remark} \label{remark:optimal} As we noted above, if  $s=1$,
  $\beta_1=0$ then we get a result analogous to the
  classical extrapolation theorem.  This is enough to motivate our
  choice of these parameters.  But in some sense this choice is also
  optimal.

To see this for $\beta_1$, we will construct
  power weights that satisfy the boundedness conditions on the maximal
  operator in~\eqref{eqn:M-constraints}.   By
  Remark~\ref{remark:power-wt} above, if $\pp\in LH$ and $0\leq a <
  n/p_+$, then $w(x) = |x|^{-a} \in A_\pp$.  Using this, we get
from \eqref{eqn:M-constraints} that $w^{\alpha_1-\beta_1} \in
A_{\pp/\alpha_1}$ and $w^{\beta_2}\in A_{\pp/s}$.  Assume that
$\alpha_1\geq \beta_1$.  Then the weight $|x|^{-a}$ satisfies these
inclusions if
\[ a(\alpha_1-\beta_1)< \frac{\alpha_1 n}{p_+}, \quad
a(s+\beta_1(p_0-1)) < \frac{s n}{p_+}. \]
Clearly, we get the same range for $a$, $a<n/p_+$, in each inequality if
$\beta_1=0$, and if $\beta_1\neq 0$ one of the ranges will be smaller
than this.  Therefore, to maximize the range of exponents we should
take $\beta_1=0$.

When $\beta_1=0$ we then have $w^{\alpha_1} \in A_{\pp/\alpha_1}$ and $w^s
\in A_{\pp/s}$.     If $\alpha_1=s$, then $s=1$ and we get the single
condition $w\in A_\pp$.   If $\alpha_1>s$, then $s<1$ and so
$\alpha_1>1$ and by Proposition~\ref{prop:Ainfty-weaker} we get that
$w^{\alpha_1} \in A_{\pp/\alpha_1}$ implies $w
\in A_{\pp}$.  If $\alpha_1<s$, then $s>1$ and we again get a
condition stronger than $w\in A_\pp$.  So we have that the
choice $s=1$ is in some sense optimal.
\end{remark}

\subsection*{Proof of Theorem \ref{thm:off-weightedvar}}

For the proof we need a few propositions.  The first gives the
relationship between Muckenhoupt $A_p$ weights and $A_{p,q}$ weights.
It was first observed in~\cite{muckenhoupt-wheeden74};
the proof follows immediately from the definition.

\begin{prop}  \label{prop:Apq-Ar}
Given $p,\,q$, $1\leq p<q < \infty$, suppose $w \in A_{p,q}$.  Then
$w^q\in A_r$ when $r=1+q/p'$.
\end{prop}

The next result is not strictly necessary to our proof, but we include
it as it is the variable exponent version of
Proposition~\ref{prop:Apq-Ar}.

\begin{prop} \label{prop:Apq-Ar-var}
Given $\pp,\, \qq\in \Pp$, $1 < p(x)
  \leq q(x) < \infty$, suppose there exists $\sigma >1$ such that
  $\frac{1}{p(x)} - \frac{1}{q(x)} = \frac{1}{\sigma'}$.  Then $w \in
  A_{\pp, \qq}$ if and only if $w^\sigma \in A_{\qq/\sigma}$.
\end{prop}

\begin{proof} First note that $\sigma\crr = \cpp$. Indeed, taking the reciprocal, we have
\[ \frac{1}{\sigma\crr}
= \frac{1}{\sigma} - \frac{1}{\sigma\rr}
= 1 - \frac{1}{\sigma'} - \frac{1}{\qq}
= 1 - \frac{1}{\pp} + \frac{1}{\qq} - \frac{1}{\qq}
= \frac{1}{\cpp}. \]
The equivalence then follows by dilation and the definition of $A_\rr$
and $A_{\pp,\qq}$:
\begin{multline*} 	|B|^{-1}\| w^\sigma \chi_B \|_{\rr} \|
  w^{-\sigma} \chi_B \|_{\crr} \\
 = 	|B|^{-1}\| w\chi_B \|_{\qq}^\sigma \| w^{-1} \chi_B \|_{\cpp}^\sigma
= \big(|B|^{\frac{1}{\sigma'}-1}\| w\chi_B \|_{\qq} \| w^{-1} \chi_B \|_{\cpp}\big)^\sigma.
\end{multline*}
\end{proof}

\medskip

To state the next result recall that given $\pp\in \Pp$, the modular
is defined by
\[ \rho_\pp(f) = \rho(f) = \int_\subRn |f(x)|^{p(x)}\,dx. \]
In the case of constant exponents, the $L^p$ norm and the modular
differ only by an exponent.  In the variable Lebesgue spaces their
relationship is more subtle as the next result shows.
For a proof see~\cite[Prop.~2.21,
Cor.~2.23]{cruz-fiorenza-book}.

\begin{prop}\label{prop:mod-norm}
Given $\pp\in \Pp$, suppose $p_+ < \infty$.  Then:
 \begin{enumerate}
	\item  $\| f \|_\pp = 1$ if and only if $\rho(f) = 1$;
	\item if $\rho(f) \leq C$, then $\| f \|_{L^{\pp}} \leq \max (C^{1/p_-}, C^{1/p_+})$;
	\item if $\| f \|_{\pp} \leq C$, then $\rho(f) \leq \max( C^{p_+}, C^{p_-})$.
	\end{enumerate}
\end{prop}

\medskip

We can now prove Theorem \ref{thm:off-weightedvar}.  As we noted
above, when $\sigma=1$ Theorem~\ref{thm:off-weightedvar} reduces to
Theorem~\ref{thm:diag-weightedvar}, so we will assume $\sigma>1$.  The
proof when $p_0=1$ is more similar to that of Theorem~\ref{thm:A1var},
and so we will defer this case to below after the proof
of~Theorem~\ref{thm:A1var}.  Here we will
assume that $p_0>1$.
We will actually prove the following more general proposition.

\begin{prop}  \label{prop-offdiag}
Let $p_0,\,q_0,\,\sigma$ and exponents $\pp,\,\qq$ be as in the
statement of Theorem \ref{thm:off-weightedvar}.   Fix $\beta_1 \in \R$
and choose any $s$ such that
    \begin{equation}\label{eqn:s-off}
    q_0-q_-\left(\frac{q_0}{\sigma}-1\right)  < s < \min(q_0, q_-).
    \end{equation}
    Let $r_0 = q_0/s$, and define $\alpha_1 = s$ and
    $\beta_2 = s - \beta_1 (1 - r_0)$. Then if $w$ is a weight such
    that $M$ is bounded on $L^{\qq/s}(w^{\alpha_1 - \beta_1})$ and
    $L^{(\qq/s)'} (w^{-\beta_2})$, we have that $\|fw\|_\qq \leq
    C\|gw\|_\pp$.
  \end{prop}

  To prove Theorem \ref{thm:off-weightedvar}, we
 take $\beta_1 = 0$ and $s = \sigma$.  Since
\[ 1 - \frac{1}{\sigma}
  = \frac{1}{p_0} - \frac{1}{q_0} = \frac{1}{p_-}-\frac{1}{q_-}, \]
we have that  the second inequality in \eqref{eqn:s-off} holds. The
first inequality is equivalent to $\sigma^2-(q_0+q_-)\sigma
+q_-q_0>0$, which follows from the second inequality.  The requirement on the weight $w$ reduces to
  $M$ being bounded on $L^{\qq/\sigma}(w^\sigma)$ and
  $L^{(\qq/\sigma)'}(w^{-\sigma})$, or equivalently, $(\qq/\sigma,
  w^{\sigma})$ is an $M$-pair.

\begin{proof}
  The proof follows an outline similar to that of
  Theorem~\ref{thm:diag-weightedvar}; we will concentrate on details
  that are different.  Fix a pair $(f,g)\in \F$; as before we may
  assume without loss of generality that $0< \|f\|_{\Lq(w)},\,
  \|g\|_{\Lp(w)}<\infty$.  Moreover, if $(f,g)$
  satisfies~\eqref{result:off-weightedvar}, then so does $(\lambda f,
  \lambda g)$ for any $\lambda>0$, so without loss of generality we
  may assume that $\|g\|_{\Lp(w)}=1$.  Then by
  Proposition~\ref{prop:mod-norm} it will suffice to prove that $\| fw
  \|_{\qq} \leq C$.

Define
	\[ h_1 = \frac{f}{\| fw \|_{\qq}} + g^{\frac{\pp}{\qq}} w^{\frac{\pp}{\qq} - 1}; \]
we claim that $\| h_1 w \|_{\qq} \leq C$.    This follows from
Proposition~\ref{prop:mod-norm}:
\[ \rho_\qq(h_1w) \leq 2^{q_+} \int_\subRn
\left(\frac{f(x)w(x)}{\|fw\|_{\qq}}\right)^{q(x)}\,dx
+ 2^{q_+} \int_\subRn \big( g(x) w(x) )^{p(x)}\,dx \leq 2^{q_++1}. \]
We again use Proposition~\ref{prop:H_j} to define two operators $H_1$ and
$H_2$ as in~\eqref{eqn:H1-H2}.  Let $r_0 = q_0/s$, and fix $s$, {$0
  < s < \min(q_0, q_-)$}. Then there exists $h_2 \in L^{(\qq/s)'}$,
$\| h_2 \|_{(\qq/s)'} = 1$, such that for any $\gamma>0$,
\begin{multline*}
 \| fw \|_{\qq}^s
 \leq C\int_\subRn f^s w^s h_2  \,dx
\leq C\int_\subRn f^s H_1^{\gamma} H_1^{-\gamma} H_2 w^s  \,dx \\
 \leq C\left( \int_\subRn f^{q_0} H_1^{-\gamma (q_0/s)} H_2 w^s \,dx
 \right)^{1/r_0}
\left( \int_\subRn H_1^{\gamma r_0'} w^s H_2 \,dx \right)^{1/r_0'}
= I_1^{1/r_0} \cdot I_2^{1/r_0'}.
\end{multline*}

{We start by finding conditions to insure that $I_2$ is uniformly
  bounded. Since $h_1\in L^{\qq}(w)$ and $h_2\in L^{(\qq/s)'}$,  we
  require $H_1$ and $H_2$ to be bounded on these
  spaces.  We apply H\"older's inequality with exponent $\qq/s$ to get
    \[ I_2 \leq C \| H_1^{\gamma(q_0/s)'} w^s \|_{\qq/s} \| H_2
    \|_{(\qq/s)'}. \]
If we let $\gamma = \frac{s}{(q_0/s)'}$, then by dilation,
\[ \| H_1^{\gamma r_0'} w^s \|_{\qq/s} = \| H_1 w \|_{\qq}^s \leq 2^s
\| h_1 w \|_{\qq}^s \leq C, \qquad \| H_2 \|_{(\qq/s)'} \leq 2\| h_1
\|_{(\qq/s)'} = 2. \] }
For $H_1$ and $H_2$ to be bounded on these
spaces, by Proposition~\ref{prop:H_j} we must have that the maximal
operator satisfies
\begin{equation*} \label{eqn:max-constraint-off}
M \text{ bounded on } L^{\qq/\alpha_1} (w^{\alpha_1 - \beta_1}) \text{
  and }
L^{(\qq/s)'/\alpha_2} (w^{-\beta_2}).
\end{equation*}
For these to hold we must have that
\begin{equation} \label{eqn:sub-constraint}
q_->\alpha_1 \qquad \text{ and } \qquad (q_+/s)'>\alpha_2.
\end{equation}

It remains to estimate $I_1$; with our value of $\gamma$ we now have
that
    \[ I_1 = \int_{\R^n} f^{q_0} H_1^{-q_0/r_0'} H_2 w^s dx. \]
In order to apply~\eqref{hyp:off-weightedvar} we need to show that
$I_1$ is finite.  However, this follows from H\"older's
inequality and the above estimates for $H_1$ and $H_2$:
 \begin{multline*}
   I_1 \leq \| f \|_{L^{\qq}(w)}^{q_0} \int H_1^{q_0} H_1^{-q_0/r_0'}
   H_2 w^{{s}} \, dx \\
   = \| f \|_{L^{\qq}(w)}^{q_0} \int H_1^{{s}} H_2 w^{{s}}\, dx \leq
   \| f \|_{L^{\qq}(w)}^{q_0} \| H_1^{{s}} w^{{s}} \|_{\qq/{{s}}} \| H_2
   \|_{(\qq/{s})'} < \infty.
 \end{multline*}

 To apply our hypothesis \eqref{hyp:off-weightedvar} we need the
 weight $w_0 = (H_1^{-\gamma(q_0/s)} H_2 w^s )^{1/q_0}$ to be in $
 A_{p_0,q_0}$, or equivalently by Proposition~\ref{prop:Apq-Ar},
 $w^{q_0} = H_1^{-(q_0 - s)} H_2 w^s \in A_{r_1}$, where
\[ r_1 = 1+ \frac{q_0}{p_0'} = \frac{q_0}{\sigma}. \]
 To apply reverse factorization
 we write
\[ w^{q_0} = \left( H_1^{\frac{q_0 - s}{r_1 - 1}} w^{\beta_1} \right)^{1 - r_1} H_2 w^{s - \beta_1 (1 - r_1)}. \]
By Proposition~\ref{prop:H_j} this gives the following constraints on
$\alpha_j, \beta_j$:
\[ \alpha_1 = \frac{q_0 - s}{\frac{q_0}{\sigma} - 1} , \quad \beta_1 \in \R,
\quad \alpha_2 = 1, \quad \beta_2 = s - \beta_1 (1 - q_0/\sigma) \]
If we combine these with the constraints in~\eqref{eqn:sub-constraint}
we see that the second one there always holds and the first one holds
if
\[ s > q_0 -q_-\left(\frac{q_0}{\sigma}-1\right). \]

We can now apply \eqref{hyp:off-weightedvar}: by the definition of $h_1$ and by H\"{o}lder's inequality with respect to the undetermined exponent $\alpha(\cdot)$, we get

\begin{align*}
 I_1^{1/q_0} & \leq C\left( \int_\subRn g^{p_0} \big[ H_1^{-q_0/r_0'} w^{{s}} H_2 \big]^{p_0/q_0} \,dx  \right)^{1/p_0} \\ &\leq C\left( \int_\subRn \left( h_1^{\frac{\qq}{\pp}} w^{\frac{\qq}{\pp} - 1} \right)^{p_0} H_1^{-p_0/r_0'} H_2^{p_0/q_0} w^{{s} p_0/q_0} \,dx\right)^{1/p_0} \\
 &\leq C\bigg(\int_\subRn H_1^{p_0 (\frac{\qq}{\pp} - \frac{1}{r_0'})} H_2^{p_0/q_0} w^{p_0 (\frac{{s}}{q_0} + \frac{\qq}{\pp} - 1)} \,dx\bigg)^{1/p_0} \\
&\leq C\big\| H_1^{p_0 (\frac{\qq}{\pp} - \frac{1}{r_0'})} w^{p_0
   (\frac{\qq}{\pp} - \frac{1}{r_0'})}  \big\|_{\alpha'(\cdot)}^{1/p_0}
\| H_2^{p_0/q_0} \|_{\alpha(\cdot)}^{1/p_0} \\
& = CJ_1^{1/p_0} J_2^{1/p_0}.
 \end{align*}
If we let $\alpha(\cdot) = \frac{q_0 (\qq/s)'}{p_0}$, then by dilation
$J_2$ is uniformly bounded.
To show that $J_1$  is uniformly bounded we first note that
\[ p_0 \left( \frac{\qq}{\pp} - \frac{1}{r_0'} \right) \alpha'(\cdot)
= \qq. \]
(This is given without proof in the constant case
in~\cite[Section~3.5]{cruz-martell-perezBook}.  It follows by a
tedious but straightforward computation.   Though $r_0$ depends on $s$, the argument
only uses the fact that $\frac{1}{\pp} - \frac{1}{\qq} = \frac{1}{p_0} - \frac{1}{q_0}$, and does not depend on the value of $s$.)  Given this, then
\[ \rho_{\alpha'(\cdot)} \big(H_1^{p_0 (\frac{\qq}{\pp} - \frac{1}{r_0'})}
w^{p_0 (\frac{\qq}{\pp} - \frac{1}{r_0'})}\big)
= \int_\subRn H_1^{\qq} w^{\qq} \,dx
= \rho_{\qq} (H_1 w). \]
If we apply Proposition~\ref{prop:mod-norm} twice, since
$\|H_1\|_{\Lq(w)}\leq 2 \|H_1\|_{\Lq(w)}$ is uniformly bounded, $\rho_{\qq} (H_1 w)$ is as
well, and hence, $J_1$ is uniformly bounded.  This completes the proof.
\end{proof}

\subsection*{Proof of Theorem \ref{thm:limited-var}}

For the proof we will need a lemma due to Johnson and
Neugebauer~\cite{johnson-neugebauer91}.

\begin{lemma} \label{lemma:jn}
Given a weight $w$,  then $w\in A_p\cap RH_s$, $1<p,\,s<\infty$,  if and only if $w^s \in
A_{{\tau}}$, where ${\tau}=s(p-1)+1$.
\end{lemma}

We again prove a more general result.

\begin{prop} \label{prop-limited}
Given that the hypotheses of Theorem~\ref{thm:limited-var} hold,
suppose $\pp \in LH$ with $q_- < p_- \leq p_+ < q_+$.    Then there
exists $p_*$, $q_-<p_*<q_+$ and $s > 0$ such that
\begin{equation}\label{eqn:s-limited}
 \max\left( p_- - p_* \left( \frac{p_-}{q_-} - 1 \right),
\frac{p_* p_+}{q_+} \right) < s < \min(p_-, p_*).
\end{equation}
Define
    \[ \tau_0 = \left(\frac{q_+}{p_*}\right)' \left(\frac{p_*}{q_-} - 1\right).  \]
Let $\beta_1 \in \R$ be any constant and define
\[
\alpha_1 = q_- \left( \frac{p_* - s}{p_* - q_-} \right), \quad
\alpha_2 = \left(\frac{q_+}{p_*}\right)', \quad
\beta_2 = s\left( \frac{q_+}{p_*} \right)' - \beta_1 (1 - \tau_0). \]
Then for any weight $w$ such that
    \begin{equation}\label{pair:limited}
    w^{\alpha_1 - \beta_1} \in A_{\pp/\alpha_1}
\qquad \text{and} \qquad
w^{-\beta_2} \in A_{(\pp/s)'/\alpha_2},
    \end{equation}
we have that
    \[ \| f \|_{L^{\pp}(w)} \leq C \| g \|_{L^{\pp}(w)}, \qquad (f, g) \in \mathcal{F}. \]
  \end{prop}

\begin{remark}
It will follow from the proof that the values of $p_*$ and $s$ are
not unique.  We will also see that the $A_\pp$ conditions
in~\eqref{pair:limited} are well defined.
\end{remark}

\medskip

  To prove Theorem \ref{thm:limited-var}, note first that if we take
  $w = 1$, then \eqref{pair:limited} holds since $\pp \in LH$ and $p_->1$ implies $\pp$ has the $K_0$ condition (see Corollay 4.50 in \cite{dcu-af-crm}), and so we get the unweighted inequality \eqref{result:limited-var1}.

To prove the weighted norm inequality \eqref{result:limited-var2}, let
$p_*$ and $s$  be any values satisfying \eqref{eqn:s-limited}.  We want $\beta_2 =
0$ so that the second condition in \eqref{pair:limited} always
holds.  This is the case if we let
    \[ \beta_1 = \frac{s(q_+/p_*)'}{1 - \tau_0} = \frac{sq_-}{q_- - p_*}
    = -\frac{s\sigma}{p_*} < 0, \]
    where $\sigma = \frac{p_* q_-}{p_* - q_-}$. Then $\alpha_1 -
    \beta_1 = \sigma$, and if we let $c  = 1 - \frac{s}{p_*}$,
    the first condition in~\eqref{pair:limited} reduces to
    $w^{\sigma} \in A_{\frac{\pp}{c\sigma}}$.

\medskip

\begin{proof}
Fix an exponent $\pp\in LH$, $q_-<p_-\leq p_+<q_+$, and fix a pair
$(f, g)\in \F$.
As before, without loss of generality we may assume that
$0<\|f\|_{L^{\pp}{(w)}},\,
\|g\|_{L^{\pp}{(w)}} < \infty$.  Define $h_1 \in L^{\pp}{(w)}$,
$\|h_1\|_{L^{\pp}{(w)}} \leq 2$,  by
    \[ h_1 = \frac{f}{\| f \|_{L^{\pp}{(w)}}} + \frac{g}{\| g \|_{L^{\pp}{(w)}}}. \]
We will use Proposition~\ref{prop:H_j} to define two operators $H_1$
and $H_2$ as in~\eqref{eqn:H1-H2}.
By dilation and duality, there exists $h_2 \in L^{(\pp/s)'}$, $\|h_2\|_
{(\pp/s)'}=1$, such that
\begin{multline*}
 \| f {w} \|_{\pp}^s
 = \| f^s {w^s} \|_{\pp/s}  \leq C\int_\subRn f(x)^s h_2 (x) {w^s} \,dx
\leq C\int_\subRn f(x)^s H_1^{-\gamma} H_1^{\gamma} H_2 {w^s} \,dx \\
\leq C\left( \int_\subRn f^{p_0} H_1^{-\gamma r_0} H_2 {w^s}
                  \,dx \right)^{1/r_0}
\left( \int_\subRn H_1^{\gamma r_0'} H_2 {w^s} \,dx \right)^{1/r_0'}
= CI_1^{1/r_0} I_2^{1/r_0'},
\end{multline*}
where $r_0=p_0/s$.

We first show that  $I_2$ is uniformly bounded.   As in the proof
of Theorem~\ref{thm:diag-weightedvar}, we want
$H_1$ to be bounded on $L^{\pp}(w)$ and $H_2$ to be bounded on
$L^{(\pp/s)'}$.   Then by H\"older's inequality and dilation,
\begin{multline*}
 I_2 \leq C\|H_1^{\gamma r_0'} w^s\|_{\pp/s}\|H_2\|_{(\pp/s)'} \\
\leq C \|H_1 w\|_{\gamma r_0' \pp/s}^{\gamma r_0'} \|H_2\|_{(\pp/s)'}
\leq C\|h_1w\|_{\gamma r_0' \pp/s}^{\gamma r_0'} \|h_2\|_{(\pp/s)'}.
\end{multline*}
The last term will be uniformly bounded if we let $\gamma= s/r_0'$.
For $H_1$ and $H_2$ to be bounded on these spaces, by
Proposition~\ref{prop:H_j} we must have that
\[ M \text{ bounded on } L^{\pp/\alpha_1}(w^{\alpha_1-\beta_1})
\text{ and } L^{(\pp/s)'/\alpha_2}(w^{-\beta_2}). \]
Since $\pp\in LH$, this will be the case if
\begin{equation} \label{eqn:var-constraint}
p_->\alpha_1, \qquad (p_+/s)'>\alpha_2.
\end{equation}

\medskip

To bound $I_1$, we want to apply our
hypothesis~\eqref{hyp:limited-var}; to
do so we need to show that it is finite.  But by our assumptions on
$H_1$ and $H_2$ and the definition of $h_1$, we
have that
\begin{multline*}
 I_1
= \int_\subRn f^{p_0} H_1^{-(p_0 - s)} H_2 {w^s} \,dx
\leq \int_\subRn (\| f {w} \|_{\pp} H_1)^{p_0} H_1^{-(p_0 - s)} H_2 {w^s} \,dx \\
= \| f {w} \|_{\pp}^{p_0} \int_\subRn H_1^s H_2 {w^s} \,dx
\leq C\| f {w} \|_{\pp}^{p_0} \| H_1 {w} \|_{\pp}^s \| H_2 \|_{(\pp/s)'} < \infty.
\end{multline*}
Assume for the moment that $w_0 = H_1^{-(p_0 - s)} H_2 {w^s}
\in A_{p_0/q_-} \cap RH_{(q_+/p_0)'}$. Then
by~\eqref{hyp:limited-var} and arguing as we did in
the previous inequality, we get that
\begin{multline*}
    \int_\subRn f^{p_0} H_1^{-(p_0 - s)} H_2 {w^s} \,dx \\
\leq C\int_\subRn g^{p_0} H_1^{-(p_0 - s)} H_2 {w^s} \,dx
\leq C\| g {w} \|_{\pp}^{p_0} \int_\subRn H_1^{p_0} H_1^{-(p_0 - s)} H_2 {w^s}
\,dx
\leq C \| g {w} \|_{\pp}^{p_0}.
\end{multline*}
If we combine this with the previous estimates we get the desired
weighted norm inequality.

\medskip

We can complete the proof if our various assumptions hold.  However,
as we will see, this may not be possible with our given value of
$p_0$, and so we will introduce a new parameter $p_*$.  We first
consider the weight $w_0$.  We want $w_0 = H_1^{-(p_0 - s)} H_2
{w^s}$ to be in $A_{p_0/q_-} \cap RH_{(q_+/p_0)'}$, which by
Lemma~\ref{lemma:jn} is equivalent to $w_0^{(q_+/p_0)'} \in
A_{{\tau_0}}$, where ${\tau_{0}} = \left( \frac{q_+}{p_0}
\right)' \left( \frac{p_0}{q_-} - 1 \right) + 1$. To apply reverse
factorization, we rewrite $w_0$ as
\[
	w_0^{(q_+/p_0)'}
		= \left[ H_1^{-(p_0 - s)} H_2 {w^s} \right]^{(q_+/p_0)'}
 = \left[ H_1^{(q_-) \frac{p_0 - s}{p_0 - q_-}} {w^{\beta_1}}  \right]^{1 - {\tau_0}} H_2^{(q_+/p_0)'} {w^{s(q_+ /p_0)' - \beta_1 (1 - {\tau_{0}})}}.
\]
Therefore, by Proposition~\ref{prop:H_j} we must have that
\[  \alpha_1 = q_-\left(\frac{p_0 - s}{p_0 - q_-}\right), \quad  \beta_1 \in \R,
\quad \alpha_2 = \left(\frac{q_+}{p_0}\right)',
\quad {\beta_2 = s\left(\frac{q_+}{p_0}\right)' - \beta_1 (1 - {\tau_{0}})}. \]
If we combine this with the first constraint
in~\eqref{eqn:var-constraint} we see that we need
\[  \frac{p_-}{q_-} \left( \frac{p_0 - q_-}{p_0
    - s} \right) > 1; \]
equivalently, we must have that
\[ 	s > p_- - p_0 \left( \frac{p_-}{q_-} - 1 \right) > 0. \]
Similarly, the second constraint in~\eqref{eqn:var-constraint} implies
that we also need
\[ 	s > \frac{p_0 p_+}{q_+}. \]

However, it need not be the case that we can find such an $s$ that
also satisfies $s<\min(p_-,p_0)$.    We can overcome this problem by
changing the value $p_0$.   By limited range extrapolation in the
constant exponent case, Theorem~\ref{thm:limited-const}, we
have that our hypothesis~\eqref{hyp:limited-var} holds with $p_0$
replaced by any $p_*$, $q_-<p_*<q_+$ provided that $w_0\in A_{p_*/q_-}
\cap RH_{(q_+/p_*)'}$.

We can, therefore, repeat the entire argument above with $p_0$ replaced
by $p_*$ and we will get our desired conclusion if we can find $p_*$
and $s>0$ such that~\eqref{eqn:s-limited} holds.   (The constants
$\alpha_j,\,\beta_j,\,\tau_0$ are also redefined as in the statement of
Proposition~\ref{prop-limited}.)  This is equivalent
to the following four inequalities being true:
\begin{align*}
(1)  & \;  p_* > \frac{p_* p_+}{q_+}, \qquad
& (3) &\;  p_- >  p_- - p_* \left( \frac{p_-}{q_-} - 1 \right), \\
(2) &\;  p_* > p_- - p_* \left( \frac{p_-}{q_-} - 1 \right), \qquad
& (4) & \; p_- > \frac{p_* p_+}{q_+}.
\end{align*}
Inequalities (1) and (3) always hold. Inequality (2) is equivalent
to $p_- \left( \frac{p_*}{q_-} \right) > p_-$ which is always true.
Inequality (4) holds if $p_*$ is such that
\[ 	 q_- < p_* < \frac{q_+}{p_+} p_-<q_+; \]
such a $p_*$ exists  since $\frac{p_+}{p_-} < \frac{q_+}{q_-}$.
Therefore, we can find the desired value of $p_*$ and $s$ and this
completes the proof of Proposition~\ref{prop-limited}.
\end{proof}

\begin{remark}\label{remark:lim-reduction}

  The limited-range extrapolation theorem with constant exponents does
  not follow from Theorem~\ref{thm:limited-var}.  However, it does
  follow from Proposition~\ref{prop-limited} by choosing a different
  set of parameters.   We need to prove that if let
  $\pp = p$, $q_-<p<q_+$,
then the norm inequality $\| fw \|_{p} \leq C \| gw
  \|_{p}$ holds provided that the weight $w^p \in A_{p/q_-}\cap
  RH_{(q_+/p)'}$, which by Lemma~\ref{lemma:jn} is equivalent to
 $w^{p(q_+/p)'} \in
  A_{\tau_p}$, where $\tau_p = (\frac{q_+}{p})'
  (\frac{p}{q_-} - 1) + 1$.   Restating this condition in terms of our variable
  weight condition, we need that the norm inequality holds provided
 $w$ satisfies
\begin{equation} \label{eqn:const-req}
w^{p(q_+/p)'/\tau_p}
  \in A_{\tau_p}^{var}.
\end{equation}
(See the  comments just before Proposition~\ref{prop:Ainfty-weaker}
for this  notation.) For the two
  conditions in~\eqref{pair:limited} to reduce to this one
  requirement, we must have that:
    \begin{enumerate}
    \item The first condition must be the same as
      \eqref{eqn:const-req}.  This is the case if $\alpha_1 - \beta_1
      = p(q_+/p)'/\tau_p$, and $p/\alpha_1 = \tau_p$, or $\alpha_1
= p/\tau_p$ and $\beta_1 = \frac{p}{\tau_p}
      \left( 1 - (q_+/p)' \right)$. Therefore, $s$ and $\beta_2$
      must satisfy
    \begin{align*}
        s = \frac{p}{\tau_p} \left( 1 - \frac{p_0}{q_-} \right) + p_0, \qquad \beta_2 = s(q_+/p_0)' - \beta_1 (1 - \tau_{p_0}).
    \end{align*}
    \item The second condition must be the `dual' of
      \eqref{eqn:const-req}: i.e., $w^{-p(q_+/p)'/\tau_p} \in
      A_{\tau_p'}^{var}$. Thus we must have that
\[ \frac{(p/s)'}{\alpha_2} = \tau_p', \qquad \beta_2 =
\frac{p}{\tau_p} (q_+/p)'. \]
    \end{enumerate}
A lengthy but straightforward computation shows that these two pairs of
    values for $s$ and $\beta_2$ are exactly the same.

Finally, we also need to show that
$s$ satisfies \eqref{eqn:s-limited}: that is, with $p_- = p = p_+$, if
we have
    \[ \max\left( p - p_0 \left( \frac{p}{q_-} - 1 \right),
\frac{p_0 p}{q_+} \right) < s < \min(p, p_0). \]
This actually follows from the above computations.  First note that by
the first condition in (1), we have $s<p_0$ since $p_0>q_-$.  By the
first condition in (2) we must have $p/s>1$ for $(p/s)'$ to be
defined.    To prove the lower inequalities, it is easier to look back
to the proof to see where these come from.  The first comes from the
requirement that $p/\alpha_1>1$, which follows from the fact that in
this case we have $p/\alpha_1=\tau_p>1$.  The second condition comes
from the requirement that $(p/s)'/\alpha_2>1$, which comes from the
fact that this equal to $\tau_p'$.
\end{remark}

\begin{remark}\label{remark:limited-cases}
  The computations in the previous remark also show why our
  extrapolation theorem is stated in a way that is quite different
  from the constant exponent case.  In our reduction we need to choose
  the constants so that the two conditions on the weight in
  \eqref{pair:limited} are actually the same: i.e.,
  $\alpha_1 - \beta_1 = \beta_2$ and $(\pp/\alpha_1)' =
  (\pp/s)'/\alpha_2$. But this last equality reduces to
    \[ \pp = \frac{s\alpha_2-\alpha_1}{\alpha_2-1} = \frac{s(q_+ - q_-) + q_-(p_* - q_+)}{p_* - q_-}, 
    \]
and this can only hold if $\pp = p$ is a constant.  However, in obtaining \eqref{result:limited-var2}, we did have two separate conditions from \eqref{pair:limited}, namely $w^{\sigma} \in A_{\frac{\pp}{c\sigma}}$ and $1 \in A_{(\pp/s)'/\alpha_2}$, which always holds. It would be of
interest to find a different version of Theorem~\ref{thm:limited-var}
that did reduce immediately to the constant exponent theorem.
\end{remark}

\subsection*{Proof of Corollary \ref{cor:limited-corollary}}
Given $\delta \in (0, 1]$ we can restate our hypothesis
\eqref{hyp:limited-corollary} as follows:
\[ \int_\subRn f(x)^{2} w_0(x) \,dx \leq c \int_\subRn  g(x)^{2} w_0(x) \,dx,  \]
for all weights $w_0$ such that $w_0^{1/\delta} \in A_2$.  By
Lemma~\ref{lemma:jn} this is equivalent to $w_0 \in A_{2/q_-} \cap
RH_{(q_+/2)'}$, where $q_- = \frac{2}{1 + \delta}$ and $q_+ =
\frac{2}{1 - \delta}$.  This is the hypothesis
\eqref{hyp:limited-var} of
Theorem~\ref{thm:limited-var}, and  applying
this theorem, we get {\eqref{result:limited-corollary1} and \eqref{result:limited-corollary2}} for all
$\pp$ satisfying \eqref{cor:limited-pp}.
%


\subsection*{Proof of Theorem~\ref{thm:A1var} and
  Theorem~\ref{thm:off-weightedvar} when $p_0=1$}

To prove Theorem~ \ref{thm:A1var} we need to modify the general
approach outlined in Section~\ref{section:general}.  To see why, first
consider the proof of Theorem~\ref{thm:diag-weightedvar}.  If we take
$p_0=1$, then the proof fails, because in order to apply H\"older's
inequality we require $s<1$, but later we need the constraint $s>1$ for the
maximal operator to be bounded on $L^{\pp/\alpha_1}(w^{\alpha_1-\beta_1})$.   This
suggests that we should not use H\"older's inequality and not
introduce the operator $H_1$ (which leads to this condition on the
boundedness of the maximal operator).  We can still dualize if we take
$s=1$, and this gives us the correct exponent to apply our
hypothesis.   We can then introduce the operator $H_2$, and argue as
before to determine the appropriate values for $\alpha_2$ and
$\beta_2$.

This seem approach works for general $p_0$.  Fix $\pp\in \Pp_0$,
$p_-\geq p_0$, and $(f,g)\in \F$.  As before, we may assume without
loss of generality that $0<\|f\|_{\Lp(w)},\, \|g\|_{\Lp(w)} <\infty$.
We will use Proposition~\ref{prop:H_j} to define an operator $H_2=\Rj_2(h_2^{\alpha_2}
w^{\beta_2})^{1/\alpha_2}w^{{-}\beta_2/\alpha_2}$.   By dilation and duality, there
exists $h_2 \in L^{(\pp/p_0)'}$, $\|h_2\|_{(\pp/p_0)'}=1$, such that
\[ \|fw\|_\pp^{p_0} \leq C\int_\subRn f^{p_0} h_2 w^{p_0}\,dx
\leq C\int_\subRn f^{p_0} H_2 w^{p_0}\,dx. \]
To apply our hypothesis~\eqref{eqn:A1hyp} we need the righthand term
to be bounded.  Since $h_2\in L^{(\pp/p_0)'}$, if we assume that $H_2$
is bounded on the same space, then by H\"older's inequality and
dilation we have that
\[  \int_\subRn f^{p_0} H_2 w^{p_0}\,dx \leq
\|fw\|_\pp^{p_0}\|H_2\|_{(\pp/p_0)'} \leq 2 \|fw\|_\pp\|h_2\|_{(\pp/p_0)'} <
\infty. \]
For $H_2$ to be so bounded, we need $M$ to be bounded on
$L^{(\pp/p_0)'/\alpha_2}(w^{-\beta_2})$.   Furthermore, to apply
our hypothesis we also need $H_2 w^{p_0}\in A_1$, so we must have that
$\alpha_2=1$ and $\beta_2=p_0$.

Therefore, if $M$ is bounded on $L^{(\pp/p_0)'}(w^{-p_0})$, we have
that
\[ \int_\subRn f^{p_0} H_2 w^{p_0}\,dx \leq
C \int_\subRn g^{p_0} H_2 w^{p_0}\,dx \leq
C \|gw\|_\pp^{p_0}\|H_2\|_{(\pp/p_0)'}  \leq C \|gw\|_\pp^{p_0}. \]
This completes the proof.

\begin{remark}
We note that in this endpoint case we do not have any flexibility in
choosing our parameters:  at each stage our choice is completely
determined by the requirements of the proof.
\end{remark}

The proof of Theorem~\ref{thm:off-weightedvar} when $p_0=1$ is nearly
identical to the proof of Theorem~\ref{thm:A1var} and can be motivated
by exactly the same analysis as we made of the proof of
Theorem~\ref{thm:diag-weightedvar}.    If we apply dilation and
duality with $p_0$ replaced by $q_0$, we get
\[ \|fw\|_\qq^{q_0} \leq C\int_\subRn f^{q_0} H_2 w^{q_0}\,dx. \]
Checking the required conditions  we see that we can apply our
hypothesis if $H_2 w^{q_0} \in A_1$, which is equivalent to
$H_2^{1/q_0} w \in A_{1,q_0}$, and this follows if the maximal operator is bounded
on $L^{(\qq/q_0)'}(w^{-q_0})$.   The rest of the proof now continues
exactly as before.

\subsection*{Proof of Theorem~\ref{thm:Ainfty-extrapolvar} and
  Proposition~\ref{prop:Ainfty-weaker}}

We could prove Theorem~\ref{thm:Ainfty-extrapolvar} by an analysis
similar to that used to prove Theorem~\ref{thm:A1var}.  However, we
can also derive it directly from this result using the connection
between $A_1$ and $A_\infty$ extrapolation
(cf.~\cite[Proposition~3.20]{cruz-martell-perezBook}).   Fix $\pp$ and
$s\leq p_-$ as in our hypotheses.  Then by
Theorem~\ref{thm:Ainfty-extrapol}, we have that \eqref{eqn:Ainfty-hyp}
holds with $p_0$ replaced by $s$ and for any $w_0\in A_\infty$.  In
particular, we can take $w_0\in A_1$, and this gives us the
hypothesis~\eqref{eqn:A1hyp} in Theorem~\ref{thm:A1var} with $p_0$
replaced by $s$.  The desired conclusion now follows from this result.

\medskip

 Finally, we prove Proposition~\ref{prop:Ainfty-weaker}.  Fix a ball
$B$.  Define the exponent function $\rr= \frac{1}{1-s}$.  Then it is
immediate that
\[ \frac{1}{(\pp/s)'} = \frac{s}{\cpp} + \frac{1}{\rr}. \]
Therefore, by dilation and the generalized H\"older's
inequality~\cite[Corollary~2.28]{cruz-fiorenza-book},
\begin{multline*}
 |B|^{-1} \|w^s\chi_B\|_{\pp/s}\|w^{-s}\chi_B\|_{(\pp/s)'}
\leq
|B|^{-1}\|w\chi_B\|_{\pp}^s\|w^{-s}\chi_B\|_{\cpp/s}\|\chi_B\|_\rr \\
= |B|^{-1} \|w\chi_B\|_{\pp}^s\|w^{-1}\chi_B\|_{\cpp}^s|B|^{1-s}
\leq [w]_{A_\pp}^s.
\end{multline*}
Since this is true for all $B$, $w^s\in A_{\pp/s}$.

\bibliographystyle{plain}
\bibliography{weightedvar}

\end{document}